\documentclass[11pt,reqno]{amsart}
\usepackage{graphicx}
\usepackage{amssymb,amsmath}
\usepackage{amsthm}
\usepackage{color,graphicx}
\usepackage{hyperref}
\usepackage{color}
\usepackage{epstopdf}

\newcommand{\be}{\begin{equation}}

\newcommand{\ee}{\end{equation}}

\newcommand{\R}{{\mathbb R}}

\newcommand{\ve}{{\varepsilon}}

\numberwithin{equation}{section}
\numberwithin{figure}{section}

\newtheorem{theorem}{Theorem}[section]
\newtheorem{proposition}[theorem]{Proposition}
\newtheorem{remark}[theorem]{Remark}

\newtheorem{corollary}[theorem]{Corollary}
\newtheorem{definition}[theorem]{Definition}

\begin{document}
\vglue-1cm \hskip1cm
\title[Periodic Waves for the Log-KdV Equation]{Orbital Stability of Periodic Traveling-wave solutions for the Log-KdV Equation}

\begin{center}



\maketitle

{\bf F\'abio Natali}

{Departamento de Matem\'atica - Universidade Estadual de Maring\'a\\
Avenida Colombo, 5790, CEP 87020-900, Maring\'a, PR, Brazil.}\\
{ fmnatali@hotmail.com}

\vspace{3mm}

{\bf Ademir Pastor}

{ IMECC-UNICAMP\\
Rua S\'ergio Buarque de Holanda, 651, CEP 13083-859, Campinas, SP,
Brazil.  } \\{ apastor@ime.unicamp.br}

\vspace{3mm}

{\bf Fabr\'icio Crist\'ofani }

{Departamento de Matem\'atica - Universidade Estadual de Maring\'a\\
Avenida Colombo, 5790, CEP 87020-900, Maring\'a, PR, Brazil.}\\
{ fabriciognr@gmail.com}

\end{center}

\begin{abstract}
In this paper we establish the orbital stability of
periodic waves related to the logarithmic Korteweg-de Vries
equation. Our motivation is inspired in the recent work
\cite{carles}, in which the authors established the well-posedness
and the linear stability of Gaussian solitary waves. By using the
approach put forward recently in \cite{natali1} to construct a
smooth branch of periodic waves as well as to get the spectral
properties of the associated linearized operator, we apply the
abstract theories in \cite{grillakis1} and \cite{weinstein1} to
deduce the orbital stability of the periodic traveling waves in the energy space.
\end{abstract}

\section{Introduction}

Results of well-posedness  and orbital stability of periodic
traveling waves related to the logarithmic Korteweg-de Vries
(log-KdV henceforth) equation
\be\label{logKDV}
u_t+u_{xxx}+2(u\log(|u|))_x=0,
\ee
will be shown in this manuscript.
Here, $u=u(x,t)$ designs a real-valued function of the real
variables $x$ and $t$. Equation $(\ref{logKDV})$ is a dispersive
equation and it models solitary waves in anharmonic chains with
Hertzian interaction forces (see \cite{carles}, \cite{dp},
\cite{jp}, and \cite{nesterenko}).

Depending on the boundary conditions imposed on the physical
problem, it is natural to consider  special kind of solutions called
\textit{traveling waves}, which imply a balance between the effects
of the nonlinearity and the frequency dispersion. In our context,
such waves are of the form $u(x,t)=\phi(x-\omega t)$, where
$\omega\in\mathbb{R}$ indicates the wave speed and
$\phi=\phi_{\omega}(\xi)$ is a smooth real function. By substituting
this kind of solution into $(\ref{logKDV})$ we obtain the nonlinear
second order differential equation
\be\label{travlog}
-\phi''+\omega\phi-\phi\log\phi^2+A=0,
\ee where $A$ is a constant of integration.

As is well known, if $A=0$,
$(\ref{travlog})$ admits a solution given by the Gaussian solitary
wave profile (see, for instance, \cite{Cazenave2} or
\cite{cazenave}) \be\label{soliton}
\phi_{\omega}(x)=e^{\frac{1}{2}+\frac{\omega}{2}}e^{-\frac{x^2}{2}},\
\ \ \omega\in\mathbb{R}. \ee The spectral stability related to this
solution was studied in \cite{carles}, where the authors studied the
linear operator, arising from the linearization of \eqref{logKDV}
around \eqref{soliton}, in the space $L^2(\R)$. In particular, they
shown that such an operator has a purely discrete spectrum
consisting of a double zero eigenvalue and a symmetric sequence of
simple purely imaginary eigenvalues. In addition, the associated
eigenfunctions do not decay like Gaussian functions but have
algebraic decay. Also, by using numerical approximations, they also
shown that the Gaussian initial data do not spread out and preserve
their spatial Gaussian decay in the time evolution of the linearized
equation.

It should be noted that the nonlinear orbital stability of \eqref{soliton} was also dealt with in \cite{carles}.
However, in view of the lack of uniqueness and continuous dependence, this is a conditional result.
Indeed, the authors establish the orbital stability (in the energy space) provided that uniqueness and
continuous dependence upon the data hold in a suitable subspace of $H^1(\R)$.

Our first concern in this paper is to study the Cauchy problem
\be
\left\{\begin{array}{llllll}\label{logKDV1}
u_t+u_{xxx}+(u\log(u^2))_x=0,\\
u(x,0)=u_0(x),
\end{array}\right.
\ee
where $u_0$ belongs to the periodic Sobolev space $H^1_{per}([0,L])$.
Most of our arguments will be based on the approach introduced by Cazenave \cite{cazenave} for the logarithmic Schr\"odinger equation
\be\label{NLSlog}
iu_t+\Delta u+\log(|u|^2)u=0,
\ee
where $u:\mathbb{R}^n\times\mathbb{R}\rightarrow\mathbb{C}$, is a complex-valued function. We point out that, in \cite{carles},
 the authors gave  a very simple manner of how  to use the arguments in \cite{cazenave}  in order to obtain the well-posedness
 of $(\ref{logKDV1})$ posed in $Y$ (see (\ref{Yclass})).

The logarithmic nonlinearity in $(\ref{logKDV1})$
brings a rich set of difficulties since the function
$x\in\mathbb{R}\mapsto x\log(|x|)$ is not differentiable at the
origin. The lack of smoothness  interferes, for instance, in
questions concerning the local solvability  because it is not possible to apply the contraction argument to deduce the existence of solutions. In order to get a grip on the absence of regularity, the idea is (see \cite{carles}, \cite{Cazenave3}, \cite{cazenave2}, and \cite{cazenave}) to  solve a
regularized or approximate problem.  Provided we can obtain  suitable uniform estimates
for the approximate solutions, they  converge, in a weak sense, to the solution of the original problem and it gives the existence of weak solutions in an appropriate Banach space.

Another difficulty coming from the non-smoothness of the nonlinearity,  is
the strain in establishing the uniqueness of solutions.
Indeed, energy methods, as well as, contraction arguments can not be
applied in these cases since we need, to this end, to assume that the nonlinearity
 is, at least, locally Lipschitz. It is clear that the function
$x\in\mathbb{R}\mapsto x\log(|x|)$ does not satisfy such a property at
the origin. We emphasize, however, that the uniqueness for the Cauchy problem associated with \eqref{NLSlog} was given  in \cite{cazenave}  by combining
energy estimates
with a suitable Gronwall-type inequality.

 To begin with our results, let us first observe that \eqref{logKDV} conserves (at least formally) the energy
\begin{equation}\label{conser12}
E(v)=\frac{1}{2}\int\left(v_x^2+v^2-v^2\log(v^2)\right)dx,
\end{equation}
the mass
\begin{equation}\label{conser22}
F(v)=\frac{1}{2}\int v^2dx,
\end{equation}
and the charge
\begin{equation}\label{conser32}
M(v)=\int vdx.
\end{equation}

 The above integrals  must be understood on the whole
real line or, in the periodic setting, on the interval $[0,L]$. \\
\indent Following the arguments in \cite{carles}, we obtain for any initial data $u_0\in X=H_{per}^1([0,L])$ the existence of a global solution $u\in
L^{\infty}(\mathbb{R};X)$ of $(\ref{logKDV1})$ satisfying the inequalities
\be\label{ineqconser} M(u(t))\leq M(u_0),\ \ F(u(t))\leq F(u_0),\ \
E(u(t))\leq E(u_0),\ \ \mbox{for all}\ t\in\mathbb{R}. \ee In addition, by supposing the complementary condition
\be\label{log-cond}
\partial_x
(\log|u|)\in L^{\infty}((\mathbb{R};L_{per}^{\infty}([0,L])), \ee one has that the solution $u$ exists in $C(\mathbb{R};X)$, is
unique,  satisfies $M(u(t))=M(u_0)$,
$F(u(t))=F(u_0)$ and $E(u(t))=E(u_0)$, for all $t\in\mathbb{R}$, and the data-solution map $u_0\in X\mapsto u\in C([-T,T];X)$ is continuous, for all $T>0$,.

\indent If one works with \eqref{logKDV1} on the whole real line, the energy \eqref{conser12} makes sense only for functions in the class
\begin{equation}\label{Yclass}
Y:=\{u\in H^1(\R); \;\;u^2\log|u|\in L^1(\R)\}.
\end{equation}
This lead the authors in \cite{carles} to study \eqref{logKDV1} in $Y$. On the other hand, in view of the log-Sobolev inequality (see \cite[Theorem 4.1]{dolbeault}),
\begin{equation}\label{log-sobo}
\displaystyle\int_{0}^L|v|^2\log(|v|^2)dx\leq C\left[ \int_0^Lv_x^2dx+\log\left(\dfrac{1}{L}\int_0^Lv^2dx\right)\int_0^Lv^2dx\right],
\end{equation}
such a restriction is not needed in the periodic framework. Thus the space $X$ seems to be the natural energy space to study \eqref{logKDV1}.

Our main result concerning local well-posedness will be presented in Theorem $\ref{teoWP}$ below. However, a few words of explanation are in order. The first one concerns the existence of global weak solutions.  As we have already mentioned above, this result will follow from an adaptation of the arguments in \cite{carles} in the periodic setting. Since the solution $u$ will be obtained as a weak limit of bounded sequences defined in a reflexive space, one can use Fatou's Lemma to deduce the ``conserved inequalities'' in   $(\ref{ineqconser})$. The assumption \eqref{log-cond} then  enable us to deduce the uniqueness of local solutions and consequently equalities in \eqref{ineqconser}. Another issue concerns  the uniqueness of solutions. The assumption \eqref{log-cond}, is a rather strong requirement. Note, however, that this condition holds if $u(x,t)=\phi(x-\omega t)$, where $\phi$ is an $L$-periodic and positive function.  Differently, in the non-periodic scenario, if $\phi$ is as in \eqref{soliton} then   $u(x,t)=\phi(x-\omega t)$ does not satisfy \eqref{log-cond}.

Next, we turn attention to the existence and orbital stability of periodic waves. We prove the existence of periodic solutions for \eqref{travlog}
 by using an extension of the abstract framework developed in \cite{natali1}. The approach for proving the orbital stability of such traveling waves is  divided into two basic cases. In the first one we assume $A\equiv 0$ and establish the orbital stability by a direct application of the abstract theory due to Grillakis, Shatah
and Strauss in \cite{grillakis1}. In the second one, we assume $A\neq0$ and use an adaptation of the arguments  in \cite{grillakis1} to deduce the orbital stability of a smooth surface $(\omega,A)\in \mathcal{O}\mapsto\psi_{(\omega,A)}$ of $L$-periodic travelling waves. Some arguments in our approach were borrowed from \cite{johnson}, where, following close the arguments in \cite{bona2} and \cite{grillakis1}, the author have established a general criterion to obtain the  orbital stability of periodic waves associated with the generalized Korteweg-de Vries equation (gKdV henceforth),
\be\label{gkdv}
  u_t+f(u)_x+u_{xxx}=0,
  \ee
  where $f:\mathbb{R}\rightarrow\mathbb{R}$ is a smooth real-valued function,

One may note that the functional $E$
is not smooth at the origin. Nevertheless, as we will see below, our
periodic waves are strictly positive or negative. Thus, at least in a
neighborhood of such waves, $E$ is smooth and this allows us to use
the abstract theories mentioned above.

As is well known, there are two  key ingredients in the nonlinear stability theory. The first one is, for a fixed $L>0$, the existence  of an open
set $\mathcal{O}\subset\mathbb{R}^n$, and a smooth branch
$\mu\in \mathcal{O}\mapsto \psi_{\mu}$,  such that
$\psi_{\mu}$ is $L$-periodic and solves $(\ref{travlog})$, for all $\mu\in \mathcal{O}$.
In our case, we will see that $L$ belongs to a convenient open interval contained in $\mathbb{R}$ and $\mu$ is either $\omega$ (in the case $A=0$) or $(\omega,A)$ (in the case $A\neq0$).

The second ingredient is the knowledge of the non-positive spectrum of the linearized operator around the periodic traveling wave in question. Usually, this  turns out to be a Hill's operator as
\begin{equation}\label{opera12}
\mathcal{L}=-\partial_x^2+g'(\mu,\psi_{\mu}).
\end{equation}
 Here, we use the recent theory developed in \cite{natali1}, where the authors presented
a new method based on the classical Floquet theorem  to establish
a characterization of the first three eigenvalues of $\mathcal{L}$
by knowing one of its eigenfunctions. The main point  is
that it is not necessary to know an explicit solution
of a general nonlinear differential
equation of the form
\begin{equation}\label{geneq}
-\psi_{\mu}''+g(\mu,\psi_{\mu})=0.
\end{equation}
In addition, it is possible to decide that the eigenvalue
\textit{zero} is simple also without knowing an explicit
solution of $(\ref{geneq})$. In particular, we show the operator
$\mathcal{L}$ appearing in our context has only one negative eigenvalue which is
simple and zero is a simple eigenvalue with
$$
\ker(\mathcal{L})=\textrm{span}\{\psi_{\mu}'\}.
$$
Moreover, the remainder of the spectrum is discrete and bounded away
from zero. For the precise statements we refer the reader to Section 4.

\indent This paper is organized as follows: In Section 2 is proved the well-posedness and the existence of
conservation laws related to the model $(\ref{logKDV})$. Existence of periodic waves for an ODE like \eqref{geneq} is treated in Section 3. In Section 4 we apply the method developed in Section 3 to study the existence of periodic travelling waves for \eqref{travlog}. The orbital stability of such waves is then established.

\section{Well Posedness Results - Verbatim of \cite{carles}}

In this section we sketch the proof of the local well-posedness theory by using the leading arguments in \cite{carles} and
\cite{Cazenave2} (see also \cite{Cazenave3} and \cite{cazenave}). The main different point here is that instead of proving the well-posedness in a class similar to that in \eqref{Yclass}, we establish our result in the whole energy space $X$. Here and throughout this section, $L>0$ will be a fixed number representing the period of the function in question. The next theorem gives a result on the existence of (weak) solutions  to \eqref{logKDV1}  in the energy space $X$.

\begin{theorem}\label{teoWP}
For any $u_0\in X$, there exists a global solution $u\in
L^{\infty}(\mathbb{R};X)$ of $(\ref{logKDV1})$ such that
\be\label{ineqconser1} M(u(t))\leq M(u_0),\ \ F(u(t))\leq F(u_0),\ \
E(u(t))\leq E(u_0),\ \ \mbox{for all}\ t\in\mathbb{R}. \ee Moreover,
if \be\label{log-cond1}
\partial_x
(\log|u|)\in L^{\infty}(\mathbb{R};L_{per}^{\infty}([0,L])), \ee then the solution $u$ exists in $C(\mathbb{R};X)$, is
unique, for all $t\in\mathbb{R}$, it satisfies $M(u(t))=M(u_0)$,
$F(u(t))=F(u_0)$ and $E(u(t))=E(u_0)$ and, for all $T>0$, the data-solution map $u_0\in X\mapsto u\in C([-T,T];X)$ is continuous.
\end{theorem}

To begin with, let us recall the following well-posedness result associated with the (generalized) KdV equation in the periodic setting.

\begin{theorem}\label{teoregKDV}
The initial-value problem
\be\label{regIVP1}
\left\{\begin{array}{lllll}
u_t+u_{xxx}+f'(u)u_x=0, \ \ \ \ t\in \mathbb{R},\\
u(x,0)=u_0(x),\ \ \ \ \ \ \ \ \ \ \ \ \ x\in[0,L],\end{array}
\right.
\ee
is locally well-posed provided $f$ is a $C^6$-function
and the initial data $u_0$ belongs to $H_{per}^s([0,L])$, $s > 1/2$. More precisely, there exist $T_0=T_0(\|u_0\|_{H^s_{per}})>0$ and a unique solution, defined in $[-T_0,T_0]$, satisfying \eqref{regIVP1} in the sense of the associated integral equation.
\end{theorem}
\begin{proof}
See Theorem 1.3 in \cite{Hu}.
\end{proof}

In addition, the smoothness of the function $f$ in Theorem $\ref{teoregKDV}$ enable
us to establish that
\be
M(u(t))=M(u_0),\ \ \ \ \ \ F(u(t))=F(u_0),\ \ \mbox{for all}\
t\in[-T_0,T_0],\label{modconserved2}
\ee
and
\be\label{modconserved1}
\widetilde{E}(u(t))=\widetilde{E}(u_0),\ \
\mbox{for all}\ t\in[-T_0,T_0],
\ee
where $\widetilde{E}$ is the modified energy, defined as,
\be\label{conserved1}
\widetilde{E}(v)=\frac{1}{2}\int_0^Lv_x^2dx-\int_0^LW(v)dx,\ \
W(v):=\int_0^vf(s)ds.
\ee
As a consequence of the above conservation laws, we deduce if $u_0$ belongs to  $H^1_{per}([0,L])$, then the solution obtained in Theorem \ref{teoregKDV} can be extended globally-in-time.

It is obvious that  $f(u)=u\log|u|$ does not satisfy the assumption in Theorem \ref{teoregKDV}. The contrivance then is to regularize the nonlinearity. To do so,  for any $\ve>0$, let us define the family of regularized nonlinearities in the form
\be\label{reglog}
f_{\ve}(u)=\left\{\begin{array}{lllll}
f(u),\ \ \ |u|\geq\ve,\\
p_{\ve}(u),\ \ |u|\leq\ve,
\end{array}\right.
\ee
where $f(u)=u\log(|u|)$ and $p_{\ve}$ is the  polynomial of degree 13 defined by
$$p_{\ve}(u):=\left(\log(\ve)-\frac{1}{2}\right)u+\sum_{i=1}^{6}\frac{a_i}{\ve^{2i}}u^{2i+1},$$
with $a_i\in\mathbb{R}$, $1\leq i\leq 6$,  determined by
using the  equality $\partial_u^kp_{\ve}(\ve) =\partial_u^kf(\ve)$, for
all $0\leq k\leq 6$.

Next, we consider the approximate Cauchy problem
\be\label{regIVP}
\left\{\begin{array}{lllll}
u_t^{\ve}+u_{xxx}^{\ve}+f_{\ve}'(u^{\ve})u_x^{\ve}=0, \ \ \ \ t>0,\\
u^{\ve}(x,0)=u_0(x),\ \ \ \ \ \ \ \ \ \ \ \ \
\end{array} \right. \ee and assume that $u_0\in
H^1_{per}([0,L])$. Theorem \ref{teoregKDV} implies the existence of
global solutions $u^\ve$ in $C(\R;H^1_{per}([0,L]))$. The remainder
of the proof follows similarly from the arguments in \cite{carles}.
Indeed, in order to pass the limit in $(\ref{regIVP})$ and proving
the existence of weak solutions associated with the original problem
$(\ref{logKDV1})$, it makes necessary to obtain uniform estimates,
independent of $\ve>0$, for the regularized solution $u^{\ve}$.
After that, by using some compactness tools, we are in position to
obtain the solution $u$ as a weak limit of the sequence $u^{\ve}$.
The uniqueness of solutions is proved once we assume that $u$
satisfies $\partial_x (\log|u|)\in
L^{\infty}(\mathbb{R};L_{per}^{\infty}([0,L]))$. Thus, the solution $u$
exists in $C(\mathbb{R};X)$, is unique
and satisfies $F(u(t))=F(u_0)$, $M(u(t))=M(u_0)$ and $E(u(t))=E(u_0)$ for all
$t\in\mathbb{R}$. The existence of the conserved quantities can be
determined by following the arguments in \cite[Theorem
3.3.9]{Cazenave2} for the general nonlinear Schr\"odinger equation.
Theorem \ref{teoWP} is thus proved.

\section{Existence of periodic traveling waves and spectral analysis - An extension of \cite{natali1}.} \label{secex}

\subsection{Existence of periodic waves}
Our purpose in this subsection is to study the existence of
periodic solutions for nonlinear ODE's written in the
form
\be
 - \phi'' + g(\mu,\phi) = 0, \label{ode}
\ee
where $g:\mathcal{P}\times\R\to\R$. It is assumed that $\mathcal{P}\subset\R^n$, $n\geq1$, is an open set, $g(\cdot,\phi)$ is smooth in $\mathcal{P}$ and $g(\mu,\cdot)$ is, at least, locally lipschitzian. This implies that a uniqueness theorem for the initial-value problem associated to \eqref{ode} holds.

The subject-matter here follows from the approach in \cite{natali1} but, for the sake of completeness, we shall give the main steps. Equation (\ref{ode}) is
conservative, and thus its solutions are contained on the level
curves of the energy
\[
\mathcal{E}(\phi, \xi) := - \frac{\xi^2}{2} + G(\mu,\phi),
\]
where $\xi=\phi'$ and  $ \partial G/ \partial \phi = g$ with
$G(\mu, 0)=0$.

We assume the following.
\begin{itemize}
\item[{\bf (H1)}] For any $\mu \in \mathcal{P}$, the function $g(\mu,\cdot)$ has two
consecutive zeros $r_1<r_2$,   such that the corresponding
equilibrium points $ (\phi, \xi) = (r_1,0)$ and $ (\phi, \xi) =
(r_2,0)$ are saddle and center, respectively.
\item[{\bf (H2)}] The level curve $\mathcal{E}(\phi, \xi)
=\mathcal{E}(r_1,0)$ contains a simple closed curve $\Gamma$
that contains $(r_2,0)$ in its interior.
\item[{\bf (H3)}] For $(\phi, \xi) $ inside  $\Gamma$
and $\mu\in \mathcal{P}$, the function $ g(\mu, \phi)$
is of class $C^1$  and $g'(\mu,r_2) <0$, where $g'$ denotes
the derivative of $g$ with respect to $\phi$.
\end{itemize}

The orbits of (\ref{ode})  inside $\Gamma$ are
periodic, turn  around $(r_2,0)$, and are contained on
the level curves $\mathcal{E}(\phi, \xi) = B$, for
$\mathcal{E}(r_1,0) < B < \mathcal{E}(r_2,0)$. Moreover, we may
suppose, without loss of generality, that the initial condition of
such solutions $(\phi(0), \phi'(0))= (\phi(0), 0)$ is inside
$\Gamma$ and $\phi(0) > r_2 $. Then, due to the symmetry of the
problem, the corresponding solutions of \eqref{ode} are periodic and even.

\begin{theorem} Under assumptions {\rm {\bf (H1)}}-{\rm {\bf (H3)}}, for every $\mu \in \mathcal{P}$, there is $L_{\mu}\in
(\alpha,+\infty)$ such that  equation (\ref{ode}) has an (even) $L_{\mu}$-periodic solution, say, $\phi_\mu$. Here, $\alpha=\alpha(\mu)$
is the period of the solutions of the linearization of
$(\ref{ode})$ at the equilibrium point $(r_2,0)$. Moreover,  $\phi_{\mu}$ and  $L_{\mu}$ are
continuously differentiable with respect to  $\mu\in\mathcal{P}$.
\label{existence}
\end{theorem}
\begin{proof} For every $\mu\in\mathcal{P}$,  the earlier arguments show that (\ref{ode}) has at least one periodic solution
$\phi_{\mu}$ with period, say, $L_{\mu}$, lying on the levels of energy $B$, for
$\mathcal{E}(r_1,0) < B <\mathcal{E}(r_2,0)$. The continuous differentiability of such solutions with respect to $\mu$ is a
consequence of the general ODE theory. Fix $\mu\in\mathcal{P}$, the period of
$\phi_{\mu}$ is given by the line integral
\be
L=\int_{\Lambda} \frac{1}{|v|}\; ds,
\label{L1}
\ee
where $\Lambda$
is the graph of $(\phi, \xi)$ in the energy level
$B$, $v(\phi, \xi) = (\xi, g(\phi))\in\R^2$ is the
vector field associated with (\ref{ode}), and $|\cdot|$ denotes the Euclidean
norm. The upper part of $\mathcal{E}(\phi,\xi) = B$ can be written
as $\xi=\sqrt{2G(\phi)-2B}$, where, for short, $G(\phi)=G(\mu,\phi)$. Thus,
\[
L = 2 \int_{b_1}^{b_2} \frac{1}{\xi(\phi)}  d\phi = 2
\int_{b_1}^{b_2} \frac{1}{\sqrt{2G(\phi) - 2 B}}  d{\phi},
\]
where $b_1, b_2$ are the roots of  $\mathcal{E}(\phi, 0) = B$. This
formula is used to compute the period $L$, but it is
inappropriate to study its differentiability with respect to
$\mu$, since $b_1, b_2$ also depend on  $\mu$ and $1/\xi(\phi)$ is singular at the end points. To overcome this, we will look
for a suitable parametrization of $\Lambda$. The linearization
of (\ref{ode}) at the equilibrium point $(r_2,0)$ is
\[
- y'' + g'(r_2)\; y = 0,
\]
where, for simplicity, we write  $g'(r_2)$ instead of $
g'(\mu,r_2)$. The  solutions of this equation are periodic
with period
\be \label{alpha}
\alpha = \frac{2\pi}{ \sqrt{-g'(r_2)}},
\ee
and their orbits are ellipses around the origin:
\[
g'(r_2)\;\frac{ y^2}{2} -\frac{y'^2}{2}   = D.
\]
For $D=-1/2$, this ellipse can be parameterized by the smooth curve
$\gamma(t)$, $t\in[0,2\pi]$, given by

\[ \gamma(t)=\left( \frac{1}{\sqrt{- g'(r_2)}} \cos{t}, \;
\sin{t}\right).
\]
The appropriate parameterization of $\Lambda$ can be obtained
through the deformation of the ellipse into the curve $\Gamma$.
Consider the system $(F,G,H)=(0,0,0)$, where
\begin{eqnarray*}
F & = & \phi - r_2 - \frac{1}{\sqrt{- g'(r_2)}}\; r \cos{t},  \\
G &=& \xi - r \sin{t}, \\
H &=& - \xi^2 + 2  G(\mu,\phi) - 2B.
\end{eqnarray*}
An application of the Implicit Function Theorem reveals one can obtain $\phi,\xi,$ and $r$ as functions of the variables $(t,\mu,B)$ and
\[
\frac{\partial \phi}{\partial t} = \frac{2 \xi r}{D \sqrt{-g(r_2)}}
\qquad \mbox{ and} \qquad \frac{\partial \xi}{\partial t} = \frac{2
r g(\phi)}{D \sqrt{-g'(r_2)}}.
\]
Therefore, from (\ref{L1}) one has that $L$ depends differentially
on the parameter $\mu$ and
\be
L_{\mu} =
\frac{2}{\sqrt{-g'(r_2)}} \int_0^{2\pi} \frac{- r}{D}\; dt.
\label{L2}
\ee
In addition, since the solutions converge uniformly on compact
intervals to $\Gamma$ (in the phase space), it is easy to see that $L_{\mu}$ goes
to infinity as $B$ goes to $\mathcal{E}(r_1,0)$.

It remains to show that $L_{\mu} \rightarrow \alpha$ as $B
\rightarrow \mathcal{E}(r_2,0)$. Since
\[
g(\phi) = g'(r_2) (\phi - r_2) + O((\phi - r_2)^2),\quad \phi - r_2 = \frac{1}{\sqrt{-g'(r_2)}}\; r \cos{t},\quad \mbox{and} \quad  \xi = r \sin{t},
\]
we obtain that $D$  satisfies
\[
D = - 2r + O((\phi - r_2)^2).
\]
Therefore, since $\phi \rightarrow r_2$, as $B \rightarrow
\mathcal{E}(r_2,0)$,
\[
L_{\mu} = \frac{2}{\sqrt{-g'(r_2)}} \int_0^{2\pi} \frac{-
r}{D}\; dt \longrightarrow \alpha = \frac{2\pi}{ \sqrt{-g'(r_2)}},
\]
as $B \rightarrow \mathcal{E}(r_2,0)$. The proof of the theorem is thus completed.
\end{proof}

\begin{remark}\label{rem1}
Theorem \ref{existence} is still true if we drop the  assumptions {\bf (H1)}-{\bf (H3)} and assume weaker conditions. In fact, it suffices to assume that $g(\mu,\cdot)$ has a zero, say, $r_2$, which is a local maximum of $G(\mu,\cdot)$. In this case, all orbits in a neighborhood of $(r_2,0)$ must be periodic orbits symmetric with respect to the $\phi$-axis in the $(\phi,\xi)$-plane (see e.g., \cite[page 178]{jack}).
\end{remark}

The proof of Theorem \ref{existence} yields an alternative formula of how
to compute the period of the solutions. In order to
apply it, we set
\begin{eqnarray}
\phi & = & r_2 + \frac{r(t)}{\sqrt{-g'(r_2)}} \; \cos{t}, \nonumber\\
\xi & = & r(t) \; \sin{t},  \label{D2}\\
D &=& \frac{ 2 g(\phi)}{\sqrt{-g'(r_2)}} \; \cos{t} - 2 \xi \;
\sin{t}, \nonumber
\end{eqnarray}
with $g(\phi)= g(\mu,  \phi)$,  $ \mu \in \mathcal{P}
$, and  $r(t)$ the solution of the first order initial-value
problem
\be \left\{ \begin{array}{l}
D r' = 2r \left( \frac{ g(\phi)}{\sqrt{-g'(r_2)}} \; \sin{t} + \xi  \cos{t} \right) \\[5mm]
 r(0) = \sqrt{-g'(r_2)}\;(\phi(0) - r_2),
 \end{array} \right.
\label{r} \ee where  $\phi(0) > r_2$, is the initial
condition of $\phi$. Thus, we have proved the following.

\begin{corollary} \label{coro1}
Let  $r(t)$ and $D(t)$ be defined as above and let $\phi_\mu$, $\mu\in\mathcal{P}$ be a periodic solution of (\ref{ode}) with initial condition $\phi(0)$. Then, the period of $\phi_\mu$  is given by
\be\label{perL}
L_{\mu} = \frac{2}{\sqrt{-g'(r_2)}}
\int_0^{2\pi} \frac{- r(t)}{D(t)}\; dt. \ee
\end{corollary}

Let $\mu\in\mathcal{P}$ be given. The next result shows if the parameter $\alpha=\alpha(\mu)$ does not depend on $\mu\in\mathcal{P}$ we can obtain $L$-periodic solutions of \eqref{ode} for any $L>\alpha$.

\begin{corollary}\label{coro2}
Assume $\alpha=\alpha(\mu)$ does not depend on $\mu\in\mathcal{P}$. Then, the period map $\mu\in\mathcal{P}\mapsto L_{\mu}\in
(\alpha,+\infty)$ obtained in Theorem $\ref{existence}$ is onto.
\end{corollary}
\begin{proof}
In fact, from Theorem $\ref{existence}$ one has that
$L_{\mu}$ is a continuously differentiable map with respect
to $\mu\in\mathcal{P}$ satisfying $L_{\mu}\rightarrow
\alpha$, as $B\rightarrow\mathcal{E}(r_2,0)$, and
$L_{\mu}\rightarrow+\infty$, as
$B\rightarrow\mathcal{E}(r_1,0)$. The result is thus proved.
\end{proof}

Our next step is to show the existence of a family,
$\mu\mapsto\psi_{\mu}$,  where each $\psi_{\mu}$ has the same
fixed period, solves equation (\ref{ode}) and  depends smoothly on
 $\mu$, for $\mu$ in a convenient open set
$\mathcal{O} \subset \mathcal{P}$. Before that, we need some basic
concepts. Let $\phi_{\mu}$ be a periodic solution of
(\ref{ode}) with period $L_{\mu}$ obtained in Theorem
\ref{existence}. Let $\mathcal{L}_{\mu}$ be the
linearized operator arising from the linearization of (\ref{ode}) at
$\phi_{\mu}$, that is,
 \be
\mathcal{L}_{\phi_{\mu}}(y):=\mathcal{L}_{\mu}(y) = -
y'' + g'(\mu, \phi_{\mu})\, y, \;\;\; \mu\in
\mathcal{P}. \label{hill}
\ee
Therefore, $\mathcal{L}_{\mu}$
is an Hill's operator, and, according to Floquet's theory (see e.g.,
\cite{magnus}), its spectrum  is formed by an unbounded sequence of
real eigenvalues
\[
\lambda_0 < \lambda_1 \leq \lambda_2 \leq \lambda_3 \leq \lambda_4\leq
\cdots\; \leq \lambda_{2n-1} \leq \lambda_{2n}\; \cdots,
\]
where equality means that $\lambda_{2n-1} = \lambda_{2n}$  is a
double eigenvalue. In addition, the spectrum of
$\mathcal{L}_{\mu}$ is characterized by the number of zeros
of the eigenfunctions in the following way: if $p$ is an eigenfunction associated to either $\lambda_{2n-1}$ or $\lambda_{2n}$, then $p$  has exactly
$2n$ zeros in the half-open
interval $[0, L_{\mu})$. In particular, the eigenfunction associated to the eigenvalue $\lambda_0$ has no zeros in $[0, L_{\mu})$.

By taking the derivative with respect to $x$ is \eqref{ode}, it is evident that $\phi'_\mu$ belongs to the kernel of the operator  $\mathcal{L}_{\mu}$, which we shall denote by $\ker(\mathcal{L}_{\mu})$. This means that
\begin{equation}\label{kerphi}
[\phi_\mu']\subseteqq \ker(\mathcal{L}_{\mu}).
\end{equation}

The next theorem proves if equality holds in \eqref{kerphi} then we can obtain the smooth family of $L$-periodic waves.

\begin{theorem}\label{existcurve}
Fix $\mu_0\in\mathcal{P}$ and let $\phi_{\mu_0}$ be an even $L_0$-periodic solution of $(\ref{ode})$ obtained in Theorem \ref{existence}, where  $L_0:=L_{\mu_0}\in(\alpha(\mu_0),+\infty)$.  Let
$\mathcal{L}_{\mu_0}$ be the linearized operator as in \eqref{hill}. If
$\ker(\mathcal{L}_{\mu_0})=[\phi_{\mu_0}']$, then
there are an open neighborhood $\mathcal{O} \subset \mathcal{P}$ of
$\mu_0$ and a family,
 $$
\mu\in \mathcal{O}\mapsto \psi_{\mu} \in H_{per,e}^2([0,L_0]),
$$ of $L_0$-periodic solutions of $(\ref{ode})$,
which  depends smoothly on $\mu\in \mathcal{O}$. In addition, $\psi_{\mu_0}=\phi_{\mu_0}$ and $\psi_\mu\to\phi_{\mu_0}$, as $\mu\to\mu_0$, in $H^2_{per}([0,L_0])$ and uniformly in $[0,L_0]$.
\end{theorem}
\begin{proof}
Here we let $H_{per,e}^2([0,L_0])$ denote the subspace of
$H_{per}^2([0,L_0])$ constituted by the even periodic functions. Let
$\mathcal{F}:\mathcal{P}\times H_{per,e}^2([0,L_0])  \rightarrow
L_{per,e}^2([0,L_0])$ be the operator defined as
\be \mathcal{F}(\mu,\psi) =
-\psi''+ g(\mu,\psi). \label{eq31}
\ee
Since $\phi_{\mu_0}$ is an even periodic solution of \eqref{ode}, it is clear that $\mathcal{F}(\mu_0,\phi_{\mu_0})=0$. Also, $\mathcal{F}$ is Fr\'echet-differentiable with respect to $\psi$ and, in particular, the derivative $\frac{\partial\mathcal{F}}{\partial\psi}(\mu_0,\phi_{\mu_0})$ is exactly the operator $\mathcal{L}_{\mu_0}$.
 By noting that  $\phi'_{\mu_0}$ is an odd function, the assumption $\ker(\mathcal{L}_{\mu_0})=[\phi_{\mu_0}']$  imply that $\mathcal{L}_{\mu_0}: H_{per,e}^2([0,L_0]) \rightarrow L_{per,e}^2([0,L_0])$  is  invertible and its inverse is bounded.
Therefore, the conclusion follows from the Implicit Function Theorem in Banach spaces
(see e.g., Theorem 15.1 and Corollary 15.1 in \cite{Deimling}) and the Sobolev embedding.
\end{proof}

\subsection{Spectral Properties}

As we already discussed in the introduction, the existence of the smooth family of $L_0$-periodic solutions in Theorem \ref{existcurve} is a first step to study the orbital stability of the traveling wave $\phi_{\mu_0}$.
As a second step, the non-positive spectrum
of $\mathcal{L}_{\psi_{\mu}}$ plays a fundamental role and, in connection, it deliveries
the major difficulty in the theory.  In our case, we study the
non-positive spectrum of $\mathcal{L}_{\psi_{\mu}}$ by studying the inertial index
$in(\mathcal{L}_{\psi_{\mu}})$, which we now introduce. To simplify notation and avoid too many technicalities, we
restrict ourselves to the study of the operators treated in our
paper. However, it is possible to obtain similar results for a large
class of self-adjoint operators and we strongly recommend the reader
to \cite{neves} and \cite{neves1} for additional informations.

\begin{definition}\label{defi12}
Let $Q$ be an $L$-periodic function. Let $\mathcal{L}$ be the Hill's operator defined in $L_{per}^2([0,L])$ with domain $D(\mathcal{L})=H_{per}^2([0,L])$ by
$$
\mathcal{L}=-\partial_x^2+Q(x).
$$
 The inertial index of $\mathcal{L}$, denoted by  $in(\mathcal{L})$, is the pair $(n, z)$, where $n$ denotes the dimension of the negative
subspace of $\mathcal{L}$ and $z$ denotes the dimension of  $\ker(\mathcal{L})$.
\end{definition}

\begin{definition}\label{defi123}
Fix $\mu_0\in\mathcal{P}$ and let $\mu\in\mathcal{O}\mapsto\psi_{\mu}$ be the smooth family of
$L_0$-periodic solutions obtained in Theorem \ref{existcurve}. The
family of linear operators
 $\mathcal{L}_{\psi_{\mu}}:=-\partial_x^2+g'(\mu,\psi_{\mu})$, $\mu\in\mathcal{O}$, is said to be \textit{isoinertial} if
 $in(\mathcal{L}_{\psi_{\mu}})=in(\mathcal{L}_{\mu_0})$, for all $\mu\in \mathcal{O}$.
\end{definition}

 The next results in this section are based on \cite{neves} and \cite{neves1}. The first
  one concerns the invariance of the inertial index with respect to the parameter $\mu\in\mathcal{O}$.

\begin{theorem}\label{teo0}
Fix $\mu_0\in\mathcal{P}$ and let $\mu\in\mathcal{O}\mapsto\psi_{\mu}$ be the smooth family of
$L_0$-periodic solutions obtained in Theorem \ref{existcurve}. Then
the family of operators  $\mathcal{L}_{\psi_{\mu}}:=-\partial_x^2+g'(\mu,\psi_{\mu})$, $\mu\in\mathcal{O}$, is
isoinertial.
\end{theorem}
\begin{proof} Since, for every $\mu\in\mathcal{P}$,  $\phi_{\mu}'$ is an eigenfunction of $\mathcal{L}_{\psi_{\mu}}$ associated with the eigenvalue $\lambda=0$,  the result follows from Theorem 3.1 in \cite{neves}.
\end{proof}

In view of Theorem \ref{teo0}, in order to obtain
$in(\mathcal{L}_{\psi_{\mu}})$, $\mu\in \mathcal{O}$,
it suffices to determine $in(\mathcal{L}_{\mu_0})$.
So, in what follows in this section, we fix $\mu_0\in\mathcal{P}$ and let $\phi_{\mu_0}$
be an even periodic solution of $(\ref{ode})$ with period $L_0=L_{\mu_0}$. As we already noted, $\phi_{\mu_0}'$ is an eigenfunction
associated with the eigenvalue $\lambda = 0$. In addition, from our construction, it has exactly two
zeros in the half-open interval $[0, L_{\mu_0})$. Thus, we
have three possibilities for the inertial index of $\mathcal{L}_{\mu_0}$:
\begin{itemize}
\item[(i)] $\lambda_1 = \lambda_2 = 0 \Rightarrow in(\mathcal{L}_{\mu_0}) = (1,2)$;
\item[(ii)] $\lambda_1 = 0 < \lambda_2 \Rightarrow in(\mathcal{L}_{\mu_0}) = (1,1)$;
\item[(iii)] $\lambda_1 < \lambda_2 = 0 \Rightarrow  in(\mathcal{L}_{\mu_0}) = (2,1)$.
\end{itemize}
The method we use to deduce $in(\mathcal{L}_{\mu_0})$  is
based on Lemma 2.1 and Theorems 2.2 and 3.1 in \cite{neves}.  Here, we assume the following.\\

\begin{itemize}
\item[{\bf (H4)}] The initial-value problem
\be \left\{
\begin{array}{l}
  -{y}'' + g'(\mu_0, \phi_{\mu_0}(x)) {y} = 0 ,\\
 {y}(0) = - \frac{1}{\phi_{\mu_0}''(0)}, \\
{y}'(0)=0,
 \end{array} \right.
\label{y}
\ee
has a unique solution, which we shall call $\bar{y}$.\\
\end{itemize}

Since $\phi_{\mu_0}'$ is an $L_0$-periodic solution of the equation in \eqref{y} and the Wronskian of $\bar{y}$ and $\phi_{\mu_0}'$ is 1, it follows from Floquet's theory (see e.g., \cite[page 4]{magnus}) that there is a constant $\theta$ (depending on $\mu_0$) such that
$$
\bar{y}(x+L_0)=\bar{y}(x)+\theta\phi_{\mu_0}'(x).
$$
By taking the derivative in this last expression and evaluating at $x=0$, we obtain
\be \label{theta}
\theta= \frac{
\bar{y}'(L_{\mu_0})}{\phi_{\mu_0}''(0)},
\ee
The simplicity of the eigenvalue zero can be characterized in the following way.

\begin{theorem} \label{teo1}
Let $\theta$ be the constant given by (\ref{theta}), then the
eigenvalue $\lambda=0$ is simple if and only if $ \theta \neq 0$.
Moreover,
\begin{itemize}
\item[(i)] $ \lambda_{1}=0$ if $\theta <0$, and
\item[(ii)] $ \lambda_{2}=0$ if $\theta > 0$.
\end{itemize}
\end{theorem}
\begin{proof}
See Section 1.2 in \cite{magnus} and  Theorem 3.1 in \cite{neves}.
\end{proof}

Combining Theorems $\ref{existcurve}$ and $\ref{teo1}$ one has the following.

\begin{corollary}\label{existcurve1}
Fix $\mu_0\in\mathcal{P}$ and let $\phi_{\mu_0}$ be an
even periodic solution of $(\ref{ode})$ with  period
$L_{\mu_0}:=L_0$. If $\theta \neq 0$
then $\ker(\mathcal{L}_{\mu_0})=[\phi_{\mu_0}']$
and Theorem $\ref{existcurve}$ holds.
\end{corollary}

The definition of isoinertial operators as above concerns to periodic solutions having the same fixed period $L_0$. Since our intention is to
prove orbital stability results of periodic waves with an arbitrary period, we need to introduce the concept of family of linear operators which are \textit{isoinertial with respect to the period}.

\begin{definition}\label{defi1234}
Let $\phi_{\mu}$, $\mu\in \mathcal{P}$, be the  $L_{\mu}$-periodic solution obtained in Theorem
\ref{existence}. The family of linear operators
 $\mathcal{L}_{\mu}$ in $(\ref{hill})$, is said to be isoinertial with respect to the period $L_{\mu}$, if
 $in(\mathcal{L}_{\mu})=in(\mathcal{L}_{\mu_0})$, for a fixed $\mu_0\in \mathcal{P}$.
\end{definition}

In what follows, we prove the family $\mathcal{L}_{\mu}$, $\mu\in\mathcal{P}$, is
isoinertial with respect to the period.

\begin{theorem} \label{teoisoL}
If $g$ satisfies assumptions {\rm {\bf (H1)}}-{\rm {\bf (H3)}} and $
g'(\mu, \phi_{\mu}(x))$ is of class $C^1$, then the
family of linear operators $\mathcal{L}_{\mu}$, $ \mu \in
\mathcal{P}$, given in \eqref{hill} is isoinertial with respect to the period.
\end{theorem}
\begin{proof}
Recall that $\mu_0\in\mathcal{P}$ is fixed and $\phi_{\mu_0}$ is an
even periodic solution of $(\ref{ode})$ with  period
$L_{\mu_0}:=L_0$. For any $\mu\in\mathcal{P}$, let $
\mathcal{M}_{\mu}: H_{per}^2([0,L_0]) \rightarrow
L_{per}^2([0,L_0])
$
be the
operator defined as
\[
\mathcal{M}_{\mu} (y) := -y'' +  \tau^2
g'(\mu,\phi_{\mu}(\tau x))\; y,
\]
where
\be  \label{tau}
 \tau = \frac{L_{\mu}}{L_0}.
\ee

Let $\eta_{\tau}$ be the dilatation that maps $L_0$-periodic
functions into $L_{\mu}$-periodic functions, that is,
\begin{eqnarray*}
\eta_{\tau}: L_{per}^2([0,L_0]) &\rightarrow& L_{per}^2([0,L_{\mu}]) \\
h(x) &\longmapsto& h\left(\frac{x}{\tau}\right).
\end{eqnarray*}
Then, it is easy to see that
\begin{eqnarray*}
\eta_{\tau}^{-1} \mathcal{L}_{\mu} \eta_{\tau} \;(y(x))
&=&  \eta_{\tau}^{-1} \mathcal{L}_{\mu}  \left(y\left(\frac{x}{\tau}\right)\right) \\
&=& \eta_{\tau}^{-1} \left( - \frac{1}{\tau^2} y'' \left(\frac{x}{\tau}\right)
+ g'(\mu, \phi_{\mu}(x)) \; y\left(\frac{x}{\tau}\right)\right) \\
&=& \frac{1}{\tau^2} \left( - y''(x) +  \tau^2 g'(\mu, \phi_{\mu}(\tau x)) y(x) \right) \\
&=& \frac{1}{\tau^2} \mathcal{M}_{\mu} (y(x)).
\end{eqnarray*}
Therefore, if $\lambda$ belongs to the resolvent set,
$\rho(\mathcal{L}_{\mu})$, of $\mathcal{L}_{\mu}$, then
\begin{eqnarray*}
\left(\mathcal{M}_{\mu} - \tau^2 \lambda I \right)^{-1} &=&
\left[ \tau^2 \left( \frac{1}{\tau^2}\mathcal{M}_{\mu} -
\lambda I \right) \right]^{-1}
= \frac{1}{\tau^2} \left( \eta_{\tau}^{-1} \mathcal{L}_{\mu} \eta_{\tau} - \lambda I \right)^{-1} \\
&=& \frac{1}{\tau^2} \; \eta_{\tau}^{-1}
\left(\mathcal{L}_{\mu} - \lambda I \right)^{-1} \eta_{\tau},
\end{eqnarray*}
that is, the resolvent sets of $ \mathcal{L}_{\mu}$ and $
\mathcal{M}_{\mu}$ satisfy the relation
\[
 \rho( \mathcal{M}_{\mu}) =  \tau^2 \rho( \mathcal{L}_{\mu}) ,
\]
where $\tau$ is given in (\ref{tau}). In particular, the operators $
\mathcal{L}_{\mu}$ and $\mathcal{M}_{\mu}$ have the
same inertial index. Now, we observe that the potential of the
operator $ \mathcal{M}_{\mu} $ is continuously differentiable
in all the variables,  and periodic with period $L_0$ for every
$\mu \in \mathcal{P}$. Therefore, Theorem 3.1 in
\cite{neves1} (see also Theorem $\ref{teo0}$) implies that
$\mathcal{M}_{\mu} $ is an isoinertial family of operators and
\[
 in( \mathcal{L}_{\mu} ) = in(\mathcal{M}_{\mu} ) = in(\mathcal{M}_{\mu_0} )=in(\mathcal{L}_{\mu_0}).
 \]
The proof of the theorem is now completed.
\end{proof}

\begin{remark}\label{obs123}
Fix $\mu_0,\mu_1\in \mathcal{P}$. Let $\phi_{\mu_0}$ and $\phi_{\mu_1}$ be the
$L_0$- and $L_1$-periodic solutions of $(\ref{ode})$ given in Theorem
$\ref{existence}$. If we may apply Theorem \ref{existcurve}, we can construct two smooth families of periodic solutions, say, $\mu\in \mathcal{O}\mapsto \psi_{\mu}\in
H_{per}^2([0,L_0])$ and $\mu\in \mathcal{O}_1\mapsto \widetilde{\psi}_{\mu}\in
H_{per}^2([0,L_0])$ such that $\psi_{\mu_0}=\phi_{\mu_0}$ and $\widetilde{\psi}_{\mu_1}=\phi_{\mu_1}$. In view of Theorems \ref{teo0} and \ref{teoisoL} we then have
$$
in(\mathcal{L}_{\psi_{\mu}})=in(\mathcal{L}_{\mu_0})=in(\mathcal{L}_{\mu_1})=in(\mathcal{L}_{\widetilde{\psi}_{\mu}}).
$$
This means that the inertial index does not depend on the curve constructed in Theorem \ref{existcurve}. In particular if $\alpha$ does not depend on $\mu$ as in Corollary \ref{coro2}, the inertial index is the same for all family $\mu\in\widetilde{\mathcal{O}}\mapsto \psi_{\mu}\in H_{per}^2([0,L])$  and any $L>\alpha$.
\end{remark}

\section{Stability of Periodic Waves for the log-kdv equation}

In this section, we use the theory put forward in last section in order to establish the existence and orbital/linear stability of periodic traveling waves for \eqref{logKDV}.

To begin with, we observe that \eqref{travlog} is of the form $(\ref{ode})$ with
\begin{equation}\label{functiong}
g(\omega,A,\phi)= \omega \phi - \phi\log\phi^2 +A.
\end{equation}

It is clear that $g$ is smooth with respect to  $(\omega,A)\in\R^2$ and locally lipschitzian in $\phi$. As we already said, we divide our analysis into two cases.\\

\subsection{First Case: $A=0$.}

\noindent \\
\indent As we have pointed out at the introduction, \eqref{travlog} admits the solitary-wave solution given in \eqref{soliton}. Thus, the dynamics associated with \eqref{travlog} is a little bit richer. Here, the function $g$ in \eqref{functiong} reduces to $g(\omega,0,
\phi):=g_\omega(\phi)=\omega \phi - \log(\phi^2) \phi$. It is easily seen that $g_\omega$ possesses three zeros, namely, $0$ and $\pm e^{\omega/2}$. In order to see that $g_\omega$ satisfies assumption {\bf (H1)} and get (positive) solutions, we take  $r_1=0$ and $r_2= e^{\omega/2}$. Since $r_1$ and $r_2$ are, respectively, local minimum and maximum of the function (see Figure \ref{figure1})
$$
G(\omega,0,\phi):=G_\omega(\phi)=\frac{\omega +1}{2} \phi^2 - \frac{1}{2} \phi^2\log\phi^2,
$$
it follows that $(r_1,0)$ and $(r_2,0)$ are, respectively, saddle and center equilibrium points (see e.g., \cite[page 179]{jack}). This shows that assumption ${\bf (H1)}$ is fulfilled, for all $\omega\in\R$.

The energy function here is given by
\[
\mathcal{E}(\phi, \xi) = - \frac{\xi^2}{2} + G_\omega(\phi) = - \frac{\xi^2}{2}
 + \frac{\omega +1}{2} \phi^2 - \frac{1}{2} \phi^2 \log\phi^2.
\]
Since the solitary-wave solution \eqref{soliton} satisfies
$$
-\frac{(\phi')^2}{2}+\frac{\omega+1}{2}\phi^2-\frac{1}{2} \phi^2\log\phi^2=0,
$$
we deduce that $\mathcal{E}(\phi, \phi')=\mathcal{E}(0,0)=0$. This means, we can take the closed
curve $\Gamma$, in assumption {\bf (H2)}, to be the orbit of the soliton
(\ref{soliton}) together with the equilibrium point
$(r_1,0)=(0,0)$ (see Figure \ref{figure1}). Since the origin belongs to $\Gamma$, it is clear that $g_\omega$ is smooth
in the region inside $\Gamma$ and $g'_\omega(r_2)=-2<0$, for all $\omega\in\R$. As a conclusion, $g_\omega$ satisfies assumptions {\bf (H1)}-{\bf (H3)} with $\mathcal{P}=\R$.
\begin{figure}[h!]
\includegraphics[scale=0.25]{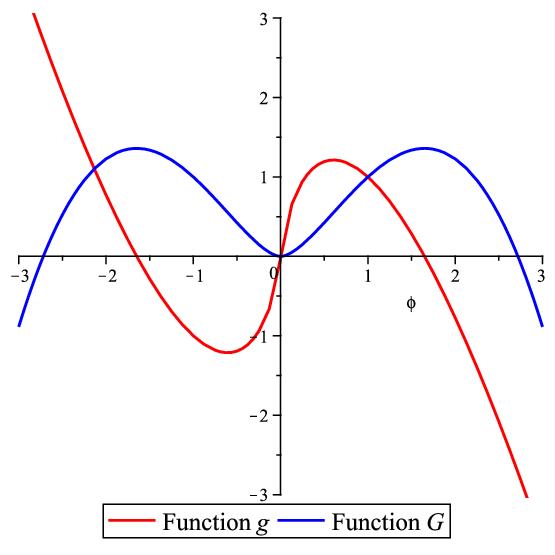}\includegraphics[scale=0.25]{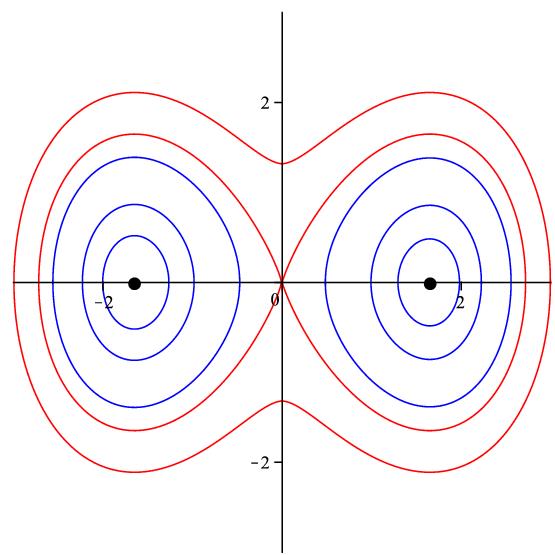}
\caption{Left: Graphs of the functions $g=g_\omega$ and $G=G_\omega$ for $\omega=1$. Right: Phase space of the equation $-\phi''+\omega\phi-\phi\log(\phi^2)=0$. The orbits in blue are those for which $\phi$ is periodic and does not change sign.}
\label{figure1}
\end{figure}

 Theorem \ref{existence} and Corollary \ref{coro2} can now be
applied to prove the existence of positive $L$-periodic solutions, where the
period $L$  ranges over the interval  $ (\alpha, +\infty)$,
with $\alpha = 2 \pi/ \sqrt{-g'_\omega(r_2)} =  \pi \sqrt{2}$. More precisely, we have.

\begin{proposition}\label{proexazero}
Let $L\in(\sqrt{2}\pi,+\infty)$ be fixed. Then, for any $\omega\in\R$, equation
\begin{equation}\label{travazero}
-\phi''+\omega\phi-\phi\log\phi^2=0
\end{equation}
possesses an $L$-periodic solution, which is even and strictly positive.
\end{proposition}

 The
initial conditions $(\phi(0), \phi'(0)) = (\phi(0),0)$ that give rise
to positive periodic  solutions are in the range  $r_2= e^{\omega/2}
< \phi(0) <  e^{(\omega +1)/2}  $, where the constant $  e^{(\omega
+1)/2} $ is the intersection of $\Gamma$ with the axis $\phi$ (in the phase space).

Note that since the $L$-periodic solutions are strictly positive the function $x\in\R\mapsto g_\omega'(\phi(x))=\omega-2-\log\phi^2(x)$ is of class $C^1$. Hence, assumption {\bf (H4)} also holds and we can apply all the results in Section \ref{secex}.
In particular, to construct a smooth family of $L$-periodic solution, one can use Theorem \ref{existcurve}. However, with the equation we have in hand,  we observe if $\phi_0$ is an $L$-periodic solution of
\begin{equation}\label{travomzero}
-\phi_0''+\phi_0\log\phi_0^2=0,
\end{equation}
then $\psi_\omega=e^{\omega/2}\phi_0$ is an $L$-periodic solution of \eqref{travazero} for any $\omega\in\R$. As a consequence, we have proved the following.

\begin{proposition}\label{proexazero1}
Let $L\in(\sqrt{2}\pi,+\infty)$ be fixed. Let $\phi_0$ be the positive $L$-periodic solution obtained in Proposition \ref{proexazero} with $\omega=0$. Then,
$$
\omega\in\R\mapsto \psi_{\omega}=e^{\omega/2}\phi_0\in H^2_{per}([0,L])
$$
is a smooth family of positive $L$-periodic  solutions for \eqref{travazero}.
\end{proposition}

It is easily seen that if $\phi$ is a solution of \eqref{travazero}, so is $-\phi$. Thus, in view of Proposition \ref{proexazero1} we may also obtain a smooth curve of negative $L$-periodic solutions. More precisely.

\begin{proposition}\label{proexazero2}
Let $L\in(\sqrt{2}\pi,+\infty)$ be fixed. Let $\phi_0$ be the positive $L$-periodic solution obtained in Proposition \ref{proexazero} with $\omega=0$. Then,
$$
\omega\in\R\mapsto \psi_{\omega}=-e^{\omega/2}\phi_0\in H^2_{per}([0,L])
$$
is a smooth family of negative $L$-periodic  solutions for \eqref{travazero}.
\end{proposition}

Attention is now turned to the orbital stability of the periodic traveling waves in Propositions \ref{proexazero1} and \ref{proexazero2}. So, in what follows in this section, we fix  $L\in(\sqrt{2}\pi,+\infty)$ and let $\psi_\omega$, $\omega\in\R$, be either a positive solution as in Proposition \ref{proexazero1} or a negative solution as in Proposition \ref{proexazero2}. Recall that the
quantities $E$ and $F$, defined in \eqref{conser12} and
\eqref{conser22}, are conserved by the flow of
\eqref{logKDV} and are invariant under the action of the group of
translations $T(s)f(\cdot)=f(\cdot+s)$, $s\in\mathbb{R}$. Note also that
functional $E$ is not smooth  at the origin on $H_{per}^1([0,L])$.
Nevertheless, the arguments above imply that $\psi_\omega^2$  is strictly positive, for any $\omega\in\R$, guaranteeing thus the smoothness
of the functional $E$ around any periodic traveling wave $\psi_\omega$. This is enough to apply
the abstract theory in \cite{grillakis1}, because
the orbital stability is determined for initial data sufficiently close to
 $\psi_\omega$.

\indent The space we shall be working with is the Hilbert space $X:=H_{per}^1([0,L])$.
 Before going into details, let us recall that periodic traveling-wave solutions are of the form $u(x,t)= \phi(x-\omega t)$, $\omega\in\R$, where $\phi$ is a solution of \eqref{travlog}. Now, we present the definition of orbital stability

\begin{definition}\label{defstab}
We say that an $L$-periodic solution $\phi$ is orbitally stable in $X$, by the periodic flow of \eqref{logKDV},  if for any $\ve>0$ there exists $\delta>0$ such that for any $u_0\in X$ satisfying $\|u_0-\phi\|_X<\delta$, the solution of \eqref{logKDV} with initial data $u_0$ exists globally and satisfies
$$
\inf_{y\in\R}\|u(\cdot,t)-\phi(\cdot+y)\|_X<\ve,
$$
for all $t\in \R$.
\end{definition}

Roughly speaking, we say that $\phi$ is orbitally stable if for any initial data close enough to $\phi$, the corresponding solution remains close enough to the orbit of $\phi$ generated by translations,
\begin{equation}\label{deforbit}
O_\phi:=\{\phi(\cdot+y);\,y\in\R\}.
\end{equation}

Define the functional $H=E+\omega F$.  Thus, in a neighborhood of  $\psi_{\omega}$,  $H$ is  smooth. This allows us to  calculate the Fr\'echet
derivative of $H$ at $\psi_{\omega}$ to deduce, from $(\ref{travlog})$, that
$\psi_{\omega}$ is a critical point of $H$, that is,
$$H'(\psi_{\omega})=(E+\omega F)'(\psi_{\omega})=- \psi_{\omega}'' + \omega \psi_{\omega} - \psi_{\omega}\log\psi_{\omega}^2
=0.$$

Also, in a neighborhood of $\psi_{\omega}$, we can rewrite equation $(\ref{logKDV})$ as an abstract Hamiltonian system, namely,
\begin{equation}\label{hamilt}
u_t=JE'(u),
\end{equation}
with $J=\partial_x$. Although  $J$ is not onto on $L^2_{per}([0,L])$, we can still apply the theory in \cite{grillakis1} because, as is well known by now, such an assumption must be imposed only for proving an instability result. As we will see below, our results show the stability of the traveling waves $\psi_\omega$.

Next, consider the linearized operator $\mathcal{L}_{\psi_\omega}:=H''(\psi_\omega)$, that is,
\be \label{linsystem}
\begin{array}{l}
\mathcal{L}_{\psi_\omega}(v)=H''(\psi_{\omega})v = - v'' + (\omega - 2  - \log\psi_{\omega}^2 )v.
\end{array}
\ee
One has that $\mathcal{L}_{\psi_\omega}$ is an unbounded operator defined on
$L_{per}^2([0,L])$ with domain $H_{per}^2([0,L])$.

Finally, we recall if (S1), (S2), and (S3) below hold, then the stability
theory presented in \cite{grillakis1} states that $\psi_\omega$ is orbitally stable.
\begin{itemize}
\item[(S1)]{There is an open interval $I\subset\mathbb{R}$ and a smooth branch of periodic solutions,
$\omega\in I\subset\mathbb{R}\mapsto\psi_{\omega}\in H_{per}^1([0,L])$.}
\item[(S2)]{ The operator $\mathcal{L}_{\psi_\omega}$, defined in \eqref{linsystem}, has only one negative eigenvalue which is simple and zero is a simple eigenvalue whose
eigenfunction is $T'(0)\psi_{\omega}=\psi_{\omega}'$, $\omega\in I$.
Moreover, the rest of the spectrum of $\mathcal{L}_{\psi_\omega}$ is positive
and bounded away from zero.}
\item[(S3)]{ If $d:I\rightarrow \mathbb{R}$ is the function defined as $d(\omega)=H(\psi_{\omega})$, then
$$
d''(\omega)=\frac{d}{d\omega}F(\psi_{\omega})>0,\ \ \  \ \ \mbox{for
all}\ \omega\in I.
$$}
\end{itemize}
\indent Finally, we are in position to prove our stability result.

\begin{theorem}\label{teostab}
Suppose that uniqueness and continuous dependence hold according
to Theorem \ref{teoWP}. Fix $L\in(\sqrt{2}\pi,+\infty)$ and let $\omega\in \R\mapsto \psi_{\omega} \in
H_{per}^2([0,L])$, be the smooth family of $L$-periodic solution given in either Proposition \ref{proexazero1} or Proposition \ref{proexazero2}. Then  $\psi_{\omega}$ is
orbitally stable in $X$ by the periodic flow of
$(\ref{logKDV})$.\end{theorem}
\begin{proof}
We let $\psi_\omega$ be any $L$-periodic wave given in Proposition \ref{proexazero1}. The case for  $\psi_\omega$ as in Proposition \ref{proexazero1} is similar.
As argued above, we need to show  that (S1), (S2) and (S3) hold.
In fact, since $\mathcal{L}_{\psi_\omega}$ is isoinertial with respect to
the period (see Remark \ref{obs123}), its inertial index can be
computed by fixing $\omega_0\in\R$ and  $L_0>\sqrt2\pi$, and calculating the inertial index of $\mathcal{L}_{\omega_0}$. For the sake of simplicity, we take  $\omega_0=0$. Thus, we see that $\phi_0(0)$ must belong to the interval
$(1,\; \sqrt{e})$. By choosing $\phi_0(0)=1.5$, $\phi_0$ satisfies the following
initial-value problem,
\begin{equation}\label{PVI3} \left\{\begin{array}{lllll}
-\phi_0''-\phi_0\log\phi_0^2=0,\\
\phi_0(0)=1.5,\\
\phi'_0(0)=0.
\end{array}\right.
\end{equation}
In order to obtain the other parameters we need for, we will use numerical arguments. In particular the  period $L_0$ of $\phi_0$ can be determined from $(\ref{perL})$ as $L_0 \approx 4.80$. By solving numerically the initial-value problems
$(\ref{PVI3})$ and \eqref{y}, we can compute the constant
$\theta$ in (\ref{theta}) as $\theta \approx -1.70$.\\

\noindent {\bf Step 1: (S1) holds.}
The existence of such a smooth branch follows  from Proposition \ref{proexazero}. Moreover, we have $I=\R$.\\

\noindent {\bf Step 2: (S2) holds.}
Since $\phi_0'$ has exactly two zeros in the interval $[0, L_0)$ and $\theta\approx -1.70<0$ we deduce from Theorem \ref{teoisoL} and Remark \ref{obs123} that
$in(\mathcal{L}_{\psi_\omega})=in(\mathcal{L}_{0})=(1,1)$ for all $\omega\in\mathbb{R}$ and all $L>\sqrt2\pi$. In addition, because $\mathcal{L}_{\psi_\omega}$ is an Hill's operator, the rest of the eigenvalues are strictly positive.\\

\noindent {\bf Step 3: (S3) holds.}
In order to conclude the orbital stability,  it remains to prove that
\begin{equation}\label{secondder1}
d''(\omega)=\displaystyle\frac{1}{2}\frac{d}{d\omega}\int_{0}^{L}\psi_{\omega}^2dx>0.
\end{equation}
But since $\psi_\omega=e^{\omega/2}\phi_0$, we immediately deduce
$$
d''(\omega)=\displaystyle\frac{1}{2}\left(\int_{0}^{L}\phi_{0}^2dx\right)\frac{d}{d\omega}e^{\omega}>0.
$$
This proves \eqref{secondder1} and concludes the proof of Theorem \ref{teostab}.
\end{proof}

\subsection{Second Case: $A\neq0$.}

\noindent \\

Next, we assume  $A\neq0$. Here the function $g$ reads as in \eqref{functiong} and
$$
G(\omega,A,\phi)= \frac{\omega +1}{2} \phi^2 - \frac{1}{2}\phi^2 \log\phi^2
+A\phi.
$$
 First of all let us take a look at the zeros of $g$. It is clear that $g(\omega,A,\phi)=0$ is equivalent to $g_\omega(\phi)=-A$, with $g_\omega(\phi)$ given in the beginning of the last subsection. A simple analysis reveals that $x_0=e^{\omega/2-1}$ and $-x_0$ are the only critical points of $g_\omega$. In addition, for all $\omega\in\R$,
$$
\lim_{x\to+\infty}g_\omega(x)=-\infty \quad \mbox{and} \quad \lim_{x\to-\infty}g_\omega(x)=+\infty.
$$
Since $g_\omega(x_0)=2e^{\omega/2-1}$ and $g_\omega(x)=-g_\omega(-x)$, we deduce three different scenarios for the zeros of $g$.

\vskip.3cm

\noindent {\bf Case 1: $|A|<2e^{\omega/2-1}$}. Here, there exist exactly three real numbers $r_0<r_1<r_2$ satisfying
$$
g_\omega(r_0)=g_\omega(r_1)=g_\omega(r_2)=-A,
$$
which means that $r_0,r_1$ and $r_2$ are three zeros of $g(\omega,A,\cdot)$, for any $\omega\in\R$ and $-2e^{\omega/2-1}<A<2e^{\omega/2-1}$. Note that  $r_1>0$ if $A<0$ and $r_1<0$ if $A>0$ (see Figure \ref{figure2}). Also, the fact that $x_0>0$ implies that $r_2>0$. In addition, because, for $\phi$ in a neighborhood of $r_2$, $g(\omega,A,\phi)>0$ if $\phi<r_2$ and $g(\omega,A,\phi)<0$ if $\phi>r_2$, it follows that $r_2$ is a local maximum of $G(\omega,A,\cdot)$. A similar analysis shows that $r_0$ is also a local maximum of $G(\omega,A,\cdot)$.
\begin{figure}[h!] \label{figure2}
\begin{center}
 \includegraphics[scale=0.25]{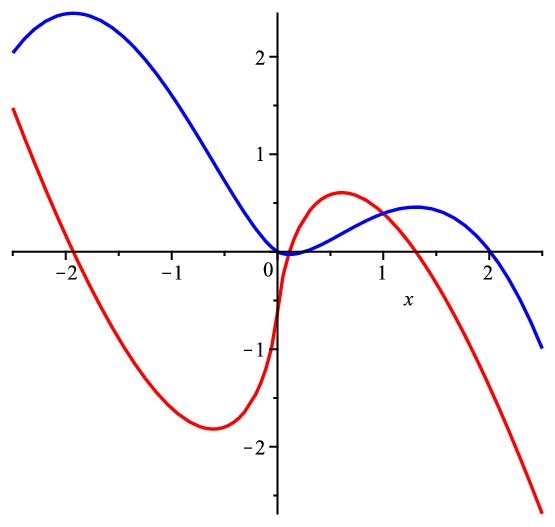}\includegraphics[scale=0.28]{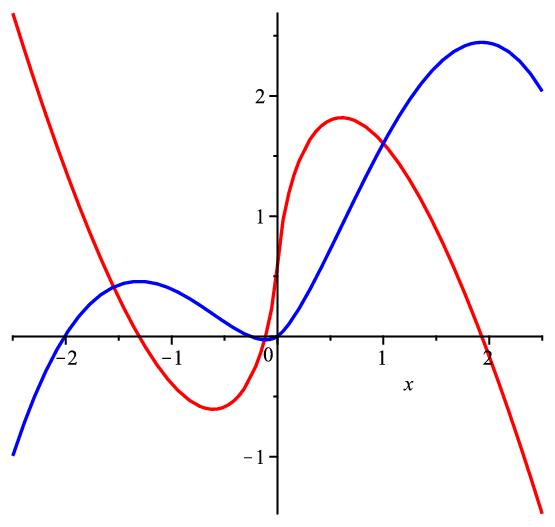}
  \caption{Left: Graphs of the functions $g$ (red) and $G$ (blue) for $\omega\in\R$ and $|A|<2e^{\omega/2-1}$, $A<0$. Right: Graphs of the functions $g$ (red) and $G$ (blue) for $\omega\in\R$ and $|A|<2e^{\omega/2-1}$, $A>0$.}
\end{center}
\end{figure}

\vskip.3cm

\noindent {\bf Case 2: $|A|=2e^{\omega/2-1}$}. In this case, there exist unique $r_1<0<r_2$ such that $g_\omega(r_1)=g_\omega(r_2)=-A$. Thus $g(\omega,A,\cdot)$ has exactly two zeros if $\omega\in\R$ and $|A|=2e^{\omega/2-1}$ (see Figure  \ref{figure3}).

\begin{figure}[h!]
 \includegraphics[scale=0.25]{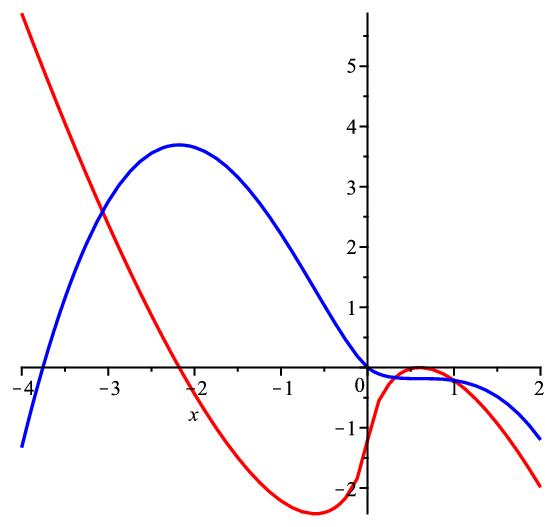}\quad \includegraphics[scale=0.25]{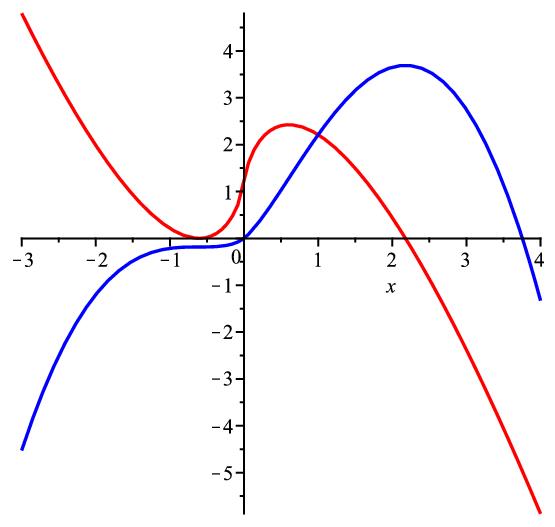}
 \caption{Left: Graphs of the functions $g$ (red) and $G$ (blue) for $\omega\in\R$ and $A=2e^{\omega/2-1}$. Right: Graphs of the functions $g$ (red) and $G$ (blue) for $\omega\in\R$ and $A=-2e^{\omega/2-1}$.}
\label{figure3}
\end{figure}

\vskip.3cm

\noindent {\bf Case 3: $|A|>2e^{\omega/2-1}$}. In this final case, there exists an unique real number $r_2$ satisfying $g_\omega(r_2)=-A$, that is, $g(\omega,A,\cdot)$ has an unique zero if $\omega\in\R$ and $|A|>2e^{\omega/2-1}$. Moreover, $r_2<0$ if $A<0$ and $r_2>0$ if $A>0$ (see Figure \ref{figure4}). Also in this case, because, for $\phi$ in a neighborhood of $r_2$, $g(\omega,A,\phi)>0$ if $\phi<r_2$ and $g(\omega,A,\phi)<0$ if $\phi>r_2$, we conclude that $r_2$ is a local maximum of $G(\omega,A,\cdot)$.
\begin{figure}[h!]
\includegraphics[scale=0.25]{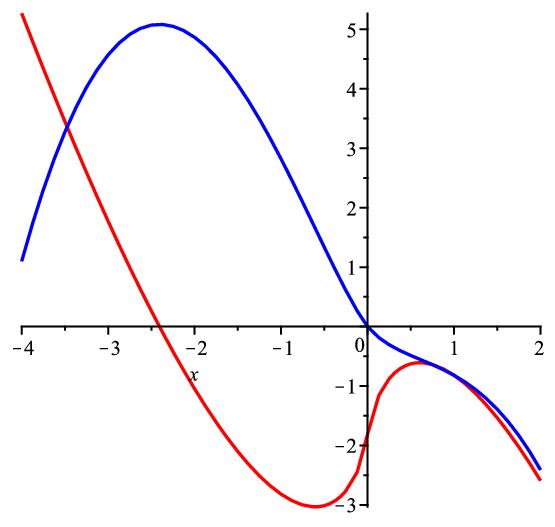}  \includegraphics[scale=0.25]{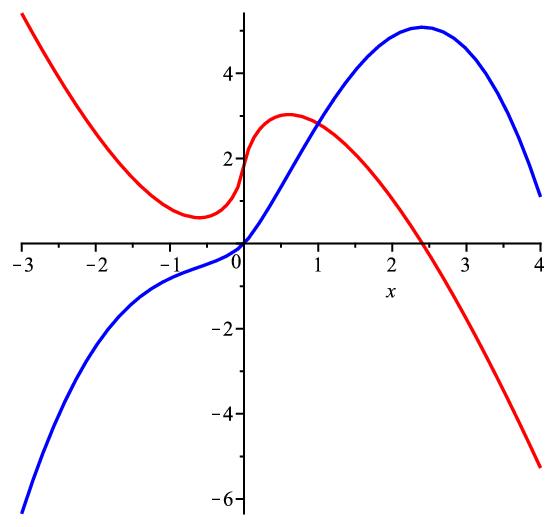}
  \caption{Left: Graphs of the functions $g$ (red) and $G$ (blue) for $\omega\in\R$ and $|A|>2e^{\omega/2-1}$, $A<0$. Right: Graphs of the functions $g$ (red) and $G$ (blue) for $\omega\in\R$ and $|A|>2e^{\omega/2-1}$, $A>0$.}
\label{figure4}
\end{figure}

\vskip.3cm

In view of the above discussion in Cases 1 and 3, if $(\omega,A)$ belongs to either
\begin{equation}\label{P1}
\mathcal{P}_1=\{(\omega,A)\in\R^2; \;\omega\in\R, \; |A|<2e^{\omega/2-1}\}
\end{equation}
or
\begin{equation}\label{P3}
\mathcal{P}_3=\{(\omega,A)\in\R^2; \;\omega\in\R, \; |A|>2e^{\omega/2-1}\},
\end{equation}
then the function $g(\omega,A,\cdot)$ always has a real zero $r_2$, for which $G(\omega,A,\cdot)$ assumes a local maximum. As a consequence of Remark \ref{rem1}, we obtain the following.

\begin{theorem}\label{teoaomegaA}
Assume that $(\omega,A)$ belongs to either $\mathcal{P}_1$ or $\mathcal{P}_3$. Then, equation \eqref{travlog} possesses an even periodic solution $\phi_{(\omega,A)}$. Moreover we have the following.
\begin{itemize}
\item[(i)] If $(\omega,A)\in\mathcal{P}_1$ with $A<0$ then all solutions that turn around $(r_2,0)$ are strictly positive and the solutions that turn around $(r_0,0)$ are strictly negative, provided they belong to a small open neighborhood of $(r_0,0)$.
\item[(ii)] If $(\omega,A)\in\mathcal{P}_1$ with $A>0$ then all solutions that turn around $(r_0,0)$ are strictly negative and the solutions that turn around $(r_2,0)$ are strictly positive, provided they belong to a small open neighborhood of $(r_2,0)$.
\item[(iii)] If $(\omega,A)\in\mathcal{P}_3$ with $A<0$ then  the solutions that turn around $(r_2,0)$ are strictly negative, provided they belong to a small open neighborhood of $(r_2,0)$.
\item[(iv)] If $(\omega,A)\in\mathcal{P}_3$ with $A>0$ then  the solutions that turn around $(r_2,0)$ are strictly positive, provided they belong to a small open neighborhood of $(r_2,0)$.
\end{itemize}

\end{theorem}

The phase spaces for $(\omega,A)$  in $\mathcal{P}_1$ or $\mathcal{P}_3$ are shown in Figures \ref{figure5} and \ref{figure6} below.

\begin{figure}[h!]
\includegraphics[scale=0.25]{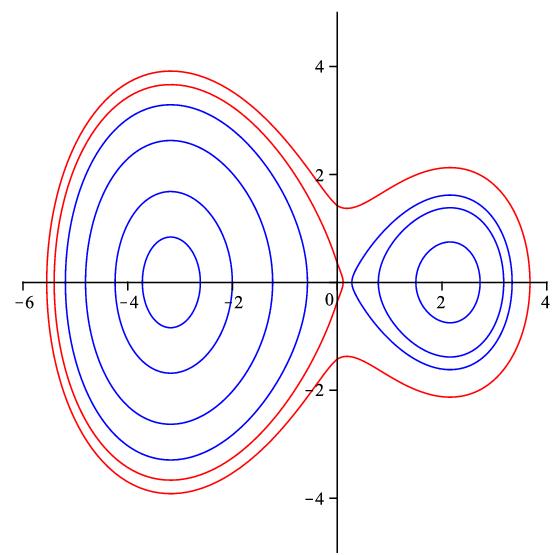}  \includegraphics[scale=0.25]{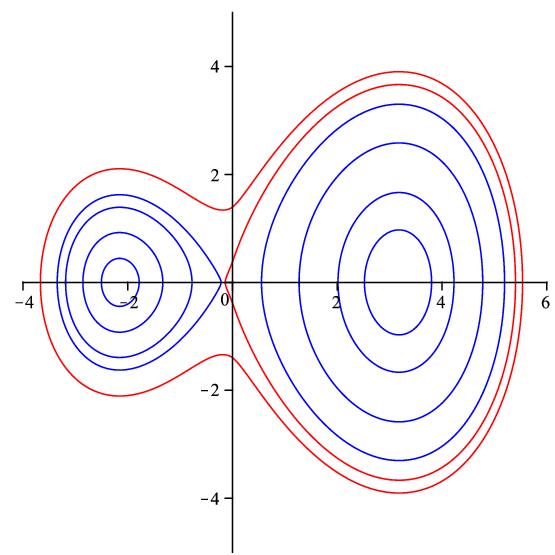}
   \caption{Left: Phase space for $|A|<2e^{\omega/2-1}$, $A<0$. Right: Phase space for $|A|<2e^{\omega/2-1}$, and $A>0$. In both cases, the orbits in blue are those for which $\phi_{(\omega,A)}$ is periodic and does not change sign.}
\label{figure5}
\end{figure}

\begin{figure}[h!]
 \includegraphics[scale=0.25]{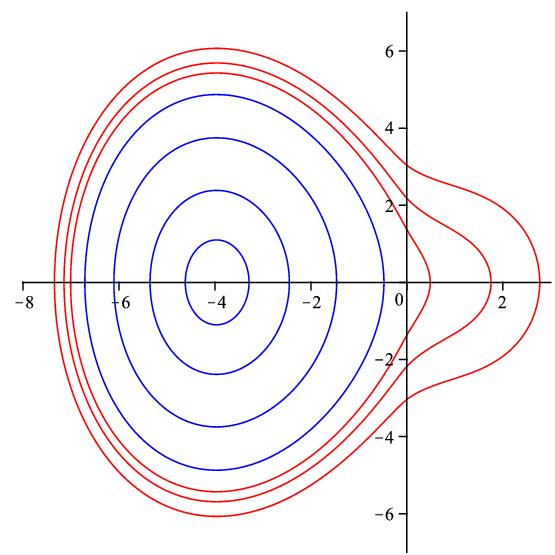}  \includegraphics[scale=0.25]{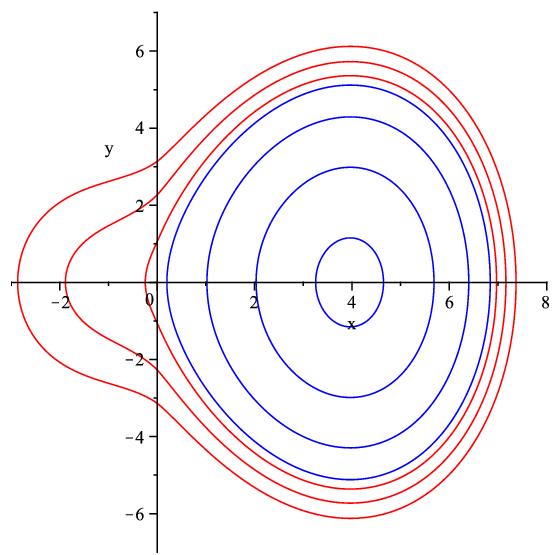}
  \caption{Left: Phase space for $|A|>2e^{\omega/2-1}$, $A<0$. Right: Phase space for $|A|>2e^{\omega/2-1}$, and $A>0$. The orbits in blue are those for which $\phi_{(\omega,A)}$ is positive.}
\label{figure6}
\end{figure}

\begin{remark}
It is clear if $|A|=2e^{\omega/2-1}$ in Case 2, we also obtain periodic solutions that do not change sign. However, in this situation $(\omega, A)$ does not belong to an open set of $\R^2$.
\end{remark}

\begin{remark}
Having disposed Theorem \ref{teoaomegaA}, a few words of explanation are in order. Here, contrary to the case where $A=0$, we have no control of how large the period of $\phi_{(\omega,A)}$ may be. There is the possibility that solutions, even those  turning around the critical center points, change its sign. Recall in the case $A=0$, we obtain $L$-periodic solutions that do not change sign for any $L>\sqrt2 \pi$.
\end{remark}

Now we can construct the smooth family of periodic solutions we need.

\begin{theorem}\label{teoexisanzero}
Fix $(\omega_0,A_0)$ according to Theorem \ref{teoaomegaA} and let $L_0$ be the period of $\phi_{(\omega_0,A_0)}$. Then
there are an open neighborhood $\mathcal{O}$ of
$(\omega_0,A_0)$ and a family,
 $$
(\omega,A)\in \mathcal{O}\mapsto \psi_{(\omega,A)} \in H_{per,e}^2([0,L_0]),
$$ of $L_0$-periodic solutions of $(\ref{ode})$,
which  depends smoothly on $(\omega,A)\in \mathcal{O}$.
\end{theorem}
\begin{proof}
To calculate the value of $\theta$ in $(\ref{theta})$ and, consequently, to apply Corollary \ref{existcurve1}, we need to obtain the value of
$\bar{y}'(L_{0})$ by solving the linear equation $(\ref{y})$ with
 $g(\omega_0,A_0,\phi)=\omega_0\phi-\log(\phi^2)\phi+A_0$. To fix ideas, let us consider $\omega_0=1$ and $A_0=1$. In this case, one has $(\omega_0,A_0)\in\mathcal{P}_1$
 and the zeros of the function $g(\omega_0,A_0,\cdot)$ are $r_0=-1$, $r_1=-0.28$ and $r_2=2.09$. To obtain strictly positive solutions with a local maximum at $x=0$, the initial condition of $\phi_{(1,1)}$ must satisfy $2.09<\phi_{(1,1)}(0)<3.51$. Collecting these informations and taking $\phi_{(1,1)}(0)=3$, we are enable to see that  $\phi=\phi_{(1,1)}$ satisfies
 \begin{equation}\label{PVI33} \left\{\begin{array}{lllll}
-\phi''+\phi-\phi\log\phi^2+1=0,\\
\phi(0)=3,\\
\phi'(0)=0.
\end{array}\right.
\end{equation}
The period $L_0$ of $\phi$ can be determined
by $(\ref{perL})$ as $L_0 \approx 4.18$. Solving numerically the initial-value problems
$(\ref{PVI33})$ and \eqref{y}, we can compute the constant
$\theta$ given by (\ref{theta}) as $\theta \approx -0.08$, which allows us to apply Corollary \ref{existcurve1}. In the table below, we present some different values of $\theta$ using the discussion established in Theorem $\ref{teoaomegaA}$.\\

\begin{center}
\begin{tabular}{|c|c|c|c|c|c|c|c|c|}
\hline
\multicolumn{5}{|c|}{\textbf{Values of $\theta$ related to $\omega_0$ and $A_0$}}\\
\hline $\omega_0$ & $A_0$ & $\phi_{(\omega_0,A_0)}(0)$ & $L_0$ &  $\theta$ \\
\hline\hline
2 & 4 & 5 & 3.49  & -0.14  \\
\hline
-1 & -1 & -0.1 & 4.32  & -0.02 \\
\hline
1 & -2 & -2 & 3.65  & -0.03  \\
\hline
-1 & -2 & -1 & 3.37  & -0.03  \\
\hline
\end{tabular}
\end{center}
\begin{tabular}{c r @{.} l}
Table 1. &
\multicolumn{1}{c}{Values of $\theta$ for different $(\omega_0,A_0)$. One has positive and negative solutions.}
\end{tabular}\\ \\
This completes the proof of the theorem.
\end{proof}

The spectral property related to the linearized operator
\be\label{opera22}\mathcal{L}=
\mathcal{L}(v)=- v'' + \big(\omega - 2  -
\log(\psi_{(\omega,A)}^2) \big)v.
\ee
 is deduced by combining the arguments in the proof of Theorem $\ref{teoexisanzero}$ with the approach treated in Section 3. More precisely.

\begin{proposition}\label{prop1}
For $(\omega,A)\in\mathcal{O}$, let
$\psi_{(\omega,A)}$ be the $L_0$-periodic solution determined in Theorem \ref{teoexisanzero}. The closed, unbounded and self-adjoint
operator $\mathcal{L}$ in $(\ref{opera22})$
defined in $L_{per}^2([0,L_0])$ with domain $H_{per}^2([0,L_0])$ has a
unique negative eigenvalue whose associated eigenfunction is even. Zero is  a simple eigenvalue with
associated eigenfunction $\psi_{(\omega,A)}'$. Moreover, the rest of the
spectrum is bounded away from zero.
\end{proposition}

Next, we present our stability result by  adapting the arguments in \cite{bona2}, \cite{grillakis1}, and \cite{johnson}. So, in what follows, we let $\psi=\psi_{(\omega_0,A_0)}$ be any $L_0$-periodic solution given in Theorem \ref{teoexisanzero}. Also,  define
$$
\eta:=\frac{\partial}{\partial\omega}\psi_{(\omega,A)}\Big|_{(\omega_0,A_0)},\ \qquad\beta:=\frac{\partial}{\partial A}\psi_{(\omega,A)}\Big|_{(\omega_0,A_0)},
$$
and set
 $$
 M_{\omega}(\psi)=\frac{\partial}{\partial\omega}\int_0^{L_0}\psi_{(\omega,A)}(x)dx\Big|_{(\omega_0,A_0)},\qquad  M_{A}(\psi)=\frac{\partial}{\partial A}\int_0^{L_0}\psi_{(\omega,A)}(x)dx\Big|_{(\omega_0,A_0)},
 $$
 and
 $$
 F_{\omega}(\psi)=\frac{1}{2}\frac{\partial}{\partial{\omega}}\int_0^{L_0}\psi_{(\omega,A)}^2(x)dx\Big|_{(\omega_0,A_0)}, \qquad F_{A}(\psi)=\frac{1}{2}\frac{\partial}{\partial{A}}\int_0^{L_0}\psi_{(\omega,A)}^2(x)dx\Big|_{(\omega_0,A_0)}.
 $$
In order to simplify the notation, the  norm and inner product in $L_{per}^2([0,L_0])$ will be denoted by  $||\cdot||$ and $\langle\cdot,\cdot\rangle$.

Before stating our main theorem we need some preliminary results. We let  $\rho$ be the semi-distance defined on the space $X$ as
\be\label{rho}
\rho(u,\psi)=\inf_{y\in\mathbb{R}}||u-\psi(\cdot+y)||_{X}.
\ee
For a given $\varepsilon>0$, we define the $\varepsilon$-neighborhood of the orbit $O_\psi$ as
\be \label{tube}U_{\varepsilon} := \{u\in X;\ \rho(u,\psi) < \varepsilon\}.\ee
We also introduce the smooth manifolds
\be\label{manifold}
\Sigma_0=\{u\in X;\ F(u)=F(\psi),\ M(u)=M(\psi)\},
\ee
and
\be\label{tau0}
\Upsilon_0=\{u\in X;\ \langle \psi,u\rangle=\langle 1,u\rangle=0\}.
\ee

The next result state that under a suitable restriction, the operator $\mathcal{L}$ is strictly positive.

\begin{proposition}\label{prop2}
Assume that there is $\Phi\in X$ such that $\langle\mathcal{L}\Phi,\varphi\rangle=0$, for all $\varphi\in \Upsilon_0$, and
\be\label{defiI}
\mathcal{I}:=\langle\mathcal{L}\Phi,\Phi\rangle<0
\ee
Then, there is a constant $c>0$ such that
$$\langle\mathcal{L}v,v\rangle\geq c||v||_{X}^2,$$
for all $v\in \Upsilon_0$ such that $\langle v,\psi'\rangle=0$.
\end{proposition}
\begin{proof}
We shall give only a sketch of the proof. From Proposition $\ref{prop1}$ one has
\be\label{decomp}L_{per}^2([0,L_0])=[\chi]\oplus [\psi']\oplus P,\ee
where $\chi$ satisfies $||\chi||=1$ and $\mathcal{L}\chi=-\lambda_0^2\chi$, $\lambda_0\neq0$.  By using the arguments  in \cite[page 278]{kato1}, we obtain that $$\langle \mathcal{L}p,p\rangle\geq c_1||p||^2,\ \ \ \ \ \mbox{for all}\ p\in H_{per}^2([0,L_0])\cap P,$$
where $c_1$ is a positive constant.

Next, from $(\ref{decomp})$, we write
$$\
\Phi=a_0\chi+b_0\psi'+p_0,\ \ \ \ \ a_0,b_0\in\mathbb{R},
$$
where $p_0\in H_{per}^2([0,L_0])\cap P$. Now, since $\psi'\in \ker (\mathcal{L})$, $\mathcal{L}\chi=-\lambda_0^2\chi$, and $\mathcal{I}<0$, we obtain
\be
\langle \mathcal{L} p_0,p_0\rangle=\langle\mathcal{L}(\Phi-a_0\chi-b_0\psi'),\Phi-a_0\chi-b_0\psi'\rangle
=\langle\mathcal{L}\Phi,\Phi\rangle+a_0^2\lambda_0^2<a_0^2\lambda_0^2.
\ee
\indent Taking $\varphi\in \Upsilon_0$ such that $||\varphi||=1$ and $\langle \varphi,\psi'\rangle=0$, we can write $\varphi=a_1\chi+p_1$, where $p_1\in X\cap P$. Thus,
\begin{equation}
0=\langle\mathcal{L}\Phi,\varphi\rangle=\langle -a_0\lambda_0^2\chi +\mathcal{L}p_0,a_1\chi+p_1\rangle
=-a_0a_1\lambda_0^2+\langle\mathcal{L}p_0,p_1\rangle.
\end{equation}
The rest of the proof runs as in \cite[Lemma 5.1]{bona2} (see also \cite[Lemma 4.4]{johnson}).

\end{proof}

 Proposition \ref{prop2} is useful to establish  the following result.

\begin{proposition}\label{prop3}
Let $E$ be the conserved quantity defined in $(\ref{conser12})$. Under the assumptions of Proposition $\ref{prop2}$ there are $\alpha>0$ and $M=M(\alpha)>0$ such that
$$E(u)-E(\psi)\geq M\rho(u,\psi)^2,$$
for all $u\in U_{\alpha}\cap\Sigma_0$.
\end{proposition}
\begin{proof}
The proof can be found in \cite[Lemma 4.6]{johnson}. So, we omit the details.
\end{proof}

\indent Finally, we present our stability result. In what follows in this section, we assume that  uniqueness and continuous dependence hold according
to Theorem \ref{teoWP}.

\begin{theorem}\label{teo2}
Let $\psi=\psi_{(\omega_0,A_0)}$ be a periodic solution given in Theorem \ref{teoexisanzero}. Assume that the matrix
$$
\mathcal{D}:=\left[\begin{array}{llll} F_{A}(\psi)\ \ M_A(\psi)\\
F_{\omega}(\psi)\ \ M_{\omega}(\psi)\end{array}\right]
$$ is invertible. If there is $\Phi\in X$ such that $\langle\mathcal{L}\Phi,\varphi\rangle=0$, for all $\varphi\in \Upsilon_0$, and $\mathcal{I}=\langle\mathcal{L}\Phi,\Phi\rangle<0$, then $\psi$ is orbitally stable in $X$ by the periodic flow of   $(\ref{logKDV})$.
\end{theorem}
\begin{proof}
Let $\alpha>0$ be the constant such that Proposition $\ref{prop3}$ holds.  Since $E$ is continuous at $\psi$, for a given $\varepsilon>0$, there exists $\delta\in (0,\alpha)$ such that if $\|u_0-\psi\|<\delta$  one has
\be\label{estepsilon}
E(u_0)-E(\psi)<M\varepsilon^2,
\ee
where $M>0$ is the constant in Proposition \ref{prop3}. We need to divide our proof into two cases.\\
\indent \textit{First case.} $u_0\in \Sigma_0$. Since $F$ and $M$ are conserved quantities, if $u_0\in\Sigma_0$ one has that $u(t)\in \Sigma_0$, for all $t\geq0$. The time continuity of the function $\rho(u(t),\psi)$ allows to choose $T>0$ such that \be\label{subalpha}\rho(u(t),\psi)<\alpha,\ \ \  \mbox{for all}\ t\in [0,T).\ee
Thus, one obtains $u(t)\in U_{\alpha}$, for all $t\in[0,T)$. Combining Proposition \ref{prop3} and $(\ref{estepsilon})$, we have
\be\label{estepsilon1}
\rho(u(t),\psi)<\varepsilon,\ \ \ \ \ \mbox{for all}\ t\in[0,T).
\ee
\indent Next, we prove that $\rho(u(t),\psi)<\alpha$, for all $t\in [0,+\infty)$, from which one concludes the orbital stability restricted to perturbations in the manifold $\Sigma_0$. Indeed, let  $T_1>0$ be the supremum of the values of $T>0$ for which $(\ref{subalpha})$ holds. To obtain a contradiction, suppose that $T_1<+\infty$.  By choosing $\varepsilon<\frac{\alpha}{2}$ we obtain, from $(\ref{estepsilon1})$,
$$
\rho(u(t),\psi)<\frac{\alpha}{2}, \ \ \ \ \ \mbox{for all}\ t\in[0,T_1).
$$
Since $t\in(0,+\infty)\mapsto\rho(u(t),\psi)$ is continuous, there is $T_0>0$ such that
$\rho(u(t),\psi)<\frac{3}{4}\alpha<\alpha$, for $t\in [0,T_1+T_0)$, contradicting the maximality of $T_1$. Therefore, $T_1=+\infty$ and the theorem  is established if $u_0\in\Sigma_0$.\\
\indent \textit{Second case.} $u_0\notin \Sigma_0$. In this case, since $\det(\mathcal{D})\neq 0$, we claim that there is  $(\omega_1,A_1)\in\mathcal{O},$  such that $F(\psi_{(\omega_1,A_1)})=F(u_0)$ and $M(\psi_{(\omega_1,A_1)})=M(u_0)$.\\
\indent In fact, since $M$ and $F$ are smooth, the Inverse Function Theorem implies the existence of $r_1,r_2>0$ such that the map
$$\begin{array}{ccc}\Gamma:B_{r_1}(\omega_0,A_0)&\longrightarrow& B_{r_2}(M(\psi),F(\psi))\\
(\omega,A)&\mapsto& (M(\psi_{(\omega,A)}),F(\psi_{(\omega,A)}))\end{array},$$

\noindent is a smooth diffeomorphism. Here, $B_{r}((x,y))$ denotes the  open ball in $\mathbb{R}^2$ centered in $(x,y)$ with radius $r>0$.  The continuity of the functionals $M$ and $V$ gives (if necessary we can take a smaller $\delta>0$)
$$|M(u_0)-M(\psi)|<\dfrac{r_2}{\sqrt2} \quad \mbox{and} \quad |F(u_0)-F(\psi)|<\dfrac{r_2}{\sqrt2},
$$
 that is, $(M(u_0),F(u_0))\in B_{r_2}(M(\psi),F(\psi))$. Since $\Gamma$ is a diffeomorphism, there is a unique $(\omega_1,A_1)\in B_{r_1}(\omega_0,A_0)$ such that $(M(u_0),F(u_0))=(M(\psi_{(\omega_1,A_1)}),F(\psi_{(\omega_1,A_1)}))$. The claim is thus proved.\\
\indent The remainder of the proof follows from the smoothness of the periodic wave with respect to the parameters, the fact that the period does not change whether $(\omega,A)\in\mathcal{O}$ and the triangle inequality.
\end{proof}

Theorem \ref{teo2} establishes the orbital stability of $\psi$ provided $det(\mathcal{D})\neq0$ and $\mathcal{I}<0$. The next proposition gives a sufficient condition to show that $\mathcal{I}<0$.

\begin{proposition}\label{propKpos}
Let $K:\R^2\to\R$ be the function defined as
$$
K(x,y)=x^2M_A(\psi)+xy(M_\omega(\psi)+F_A(\psi))+y^2F_\omega(\psi).
$$
Assume that there is $(a,b)\in\R^2$ such that $K(a,b)>0$. Then there is $\Phi\in X$ such that   $\langle\mathcal{L}\Phi,\varphi\rangle=0$, for all $\varphi\in \Upsilon_0$, and
$$
\mathcal{I}=\langle \mathcal{L}\Phi,\Phi \rangle<0.
$$
\end{proposition}
\begin{proof}
It suffices to define  $\Phi:=a\beta+b\eta$. Indeed, since $\mathcal{L}\beta=-1$ and $\mathcal{L}\eta=-\psi$, it is clear that  $\langle\mathcal{L}\Phi,\varphi\rangle=0$, for all $\varphi\in \Upsilon_0$, and
\[
\begin{split}
 \langle\mathcal{L}\Phi,\Phi \rangle&=\langle-a-b\psi,a\beta+b\eta\rangle\\
 &=-(a^2M_A(\psi)+abM_\omega(\psi)+abF_A(\psi)+b^2F_\omega(\psi))\\
 &=-K(a,b).
\end{split}
\]
The proof is thus completed.
\end{proof}

\begin{corollary}
Assume that $A_0$ is sufficiently small. Then $\psi=\psi_{(\omega_0,A_0)}$ is orbitally stable in $X$ provided $det(\mathcal{D})\neq0$.
\end{corollary}
\begin{proof}
Differentiating the equation
\be\label{qdiffi}
-\psi''+\omega\psi-\psi\log\psi^2+A=0
\ee
with respect to $\omega$, multiplying the obtained equation by $\psi$ and integrating on $[0,L_0]$ we deduce that
\begin{equation}\label{difFom}
2F_\omega(\psi)=2F(\psi)-A_0M_\omega(\psi).
\end{equation}
Since $F(\psi)>0$, we see that if $A_0$ is sufficiently small then $F_\omega(\psi)>0$. Thus, by taking $(a,b)=(0,1)$ we obtain $K(a,b)>0$. The conclusion then follows in view of Proposition \ref{propKpos} and Theorem \ref{teo2}.
\end{proof}

\begin{corollary}
Assume that $\psi>0$ and $det(\mathcal{D})\neq0$. If $A_0>0$ then there exists $(a,b)\in\R^2$ such that $K(a,b)>0$. Consequently, there exists $\Phi\in X$ such that $\mathcal{I}<0$ and $\psi$ is orbitally stable in $X$.
\end{corollary}
\begin{proof}
From Proposition \ref{propKpos} and Theorem \ref{teo2} it suffices to show the existence of $(a,b)\in\R^2$ such that $K(a,b)>0$. If $M_A(\psi)>0$ we can take $(a,b)=(1,0)$. If $F_\omega(\psi)>0$, we can take $(a,b)=(0,1)$. Assume now that $M_A(\psi)\leq0$ and $F_\omega(\psi)\leq0$. It is to be observed that since $det(\mathcal{D})\neq0$ the case $(M_A(\psi),M_\omega(\psi))=(0,0)$ is ruled out. Differentiating equation \eqref{qdiffi}
with respect to $A$, multiplying the obtained equation by $\psi$ and integrating on $[0,L_0]$ it is inferred that
\begin{equation}\label{difFA}
2F_A(\psi)=M(\psi)-A_0M_A(\psi).
\end{equation}
Taking the derivative in \eqref{difFom} with respect to $A$ and in \eqref{difFA} with respect to $\omega$ and comparing the result we see that $F_A(\psi)=M_\omega(\psi)$. Hence, the function $K$ in Proposition \ref{propKpos} reads as
$$
K(x,y)=x^2M_A(\psi)+2xyM_\omega(\psi)+y^2F_\omega(\psi).
$$
If $M_A(\psi)=F_\omega(\psi)=0$ we can take $(a,b)=(-1,1)$ or $(a,b)=(1,1)$ according to the sign of $M_\omega(\psi)$.

We now divide the rest of proof  into two cases.

\noindent {\bf Case 1.} $M_\omega(\psi)\leq0$.
Note that
\[
\begin{split}
\Delta:&=M_\omega(\psi)^2-M_A(\psi)F_\omega(\psi)\\
&=M_\omega(\psi)^2-M_A(\psi)\Big(F(\psi)-AM_\omega(\psi)\Big)\\
&=M_\omega(\psi)^2-M_A(\psi)F(\psi)-AM_A(\psi)M_\omega(\psi)>0.
\end{split}
\]
Thus, either $K(x,1)=0$ or $K(1,y)=0$ has two different real roots. In any case, this implies that there is $(a,b)\in\R^2$ such that $K(a,b)>0$.

\noindent {\bf Case 2.} $M_\omega(\psi)>0$. Note that
\[
\begin{split}
det(D)&=M_\omega(\psi)\left(\dfrac{M(\psi)}{2}-\dfrac{A}{2}M_A(\psi)\right)-M_A(\psi)\left(F(\psi)-\dfrac{A}{2}M_\omega(\psi)\right)\\
&=\dfrac{1}{2}M(\psi)M_\omega(\psi)-M_A(\psi)F(\psi)>0.
\end{split}
\]
Hence, if $M_A(\psi)<0$ we have $M_A(\psi)\,det(D)<0$. By taking $(a,b)=(M_\omega(\psi),-M_A(\psi))$, we deduce
\[
\begin{split}
K(a,b)=&=M_\omega(\psi)^2M_A(\psi)-M_\omega(\psi)M_A(\psi)\left(M_\omega(\psi)+F_A(\psi)  \right)- M_A(\psi)^2F_\omega(\psi)\\
&=-M_A(\psi)\left( M_\omega(\psi)F_A(\psi)-M_A(\psi)F_\omega(\psi) \right)\\
&=-M_A(\psi)det(\mathcal{D})>0.
\end{split}
\]
Finally suppose  $M_A(\psi)=0$. Since the case $F_\omega(\psi)=0$ has already been dealt with, we may assume $F_\omega(\psi)<0$. By taking $a=1$ and $b=-\dfrac{M_\omega(\psi)}{F_\omega(\psi)}$, we have
$$
K(a,b)=-2\dfrac{M_\omega(\psi)^2}{F_\omega(\psi)}+\dfrac{M_\omega(\psi)^2}{F_\omega(\psi)}=-\dfrac{M_\omega(\psi)^2}{F_\omega(\psi)}>0.
$$
This completes the proof of the corollary.
\end{proof}

\begin{remark}
 \indent Next table shows some values of $M_A(\psi)\det(\mathcal{D})$, $M_A(\psi)$, and $F_{\omega}(\psi)$. Although we are not able to prove analytically, numerical calculations suggest that $\det(\mathcal{D})\neq0$ and $F_{\omega}(\psi)>0$, for all $(\omega,A)\in \mathcal{P}_i$, $i=1,3$ (recall this is true in the case $A=0$). Theorem \ref{teo2} and Proposition \ref{propKpos} would imply that $\psi$ is orbitally stable in $X$ by the periodic flow of   $(\ref{logKDV})$.
 \end{remark}
\begin{center}
\begin{tabular}{|c|c|c|c|c|c|c|c|c|}
\hline
\multicolumn{7}{|c|}{\textbf{Values of $\mathcal{I}$ related to $\omega_0$ and $A_0$}}\\
\hline $\omega_0$ & $A_0$ & $\phi(0)$ & $L_0$ & $M_A(\psi)\det(\mathcal{D})$ & $M_A(\psi)$ & $F_{\omega}(\psi)$\\
\hline\hline
1 & 1 & 3 & 4.18 & -0.47 & -0.21 & 7.41\\
\hline
2 & 4 & 5 & 3.49 & 3.13 & 24.99 & 21.81\\
\hline
-1 & -1 & -0.1 & 4.32 & 0.50 & 0.44 & 1.42\\
\hline
1 & -2 & -2 & 3.65 & 2.80 & 8.99 & 7.82  \\
\hline
-1 & -2 & -1 & 3.37 & -1.65 & -0.14 & 0.52 \\
\hline
5 & 3 & 15 & 4.21 & 354.78 & 1.22 & 349.28\\
\hline
-3 & -2 & -0.5 & 2.95 & 0.18 & 0.43 & 0.22 \\
\hline
-5 & -0.1 & -0.1 & 3.76 & 0.01 & 0.45 & 0.01\\
\hline
-10 & -2 & -0.2 & 2.03 & 0.0008 & 0.20 & 0.004\\
\hline

\end{tabular}
\end{center}
\begin{center}
Table 2
\end{center}

\section*{Acknowledgement}

F. N. is supported by CNPq/Brazil. A. P. is supported by CNPq/Brazil and FAPESP/S\~ao Paulo/Brazil. F. C. is supported by CAPES/Brazil.

\end{document}